\documentclass[]{svmult}
\usepackage[utf8]{inputenc}
\usepackage{tikz}
\usepackage{tikz-cd}
\usetikzlibrary{patterns.meta}
\usepackage{newtxmath}
\usepackage{graphicx}
\usepackage{xcolor}
\usepackage{algpseudocode}
\usepackage[linesnumbered,lined,boxed,commentsnumbered]{algorithm2e}
\usepackage{mathtools}
\usepackage{enumitem}
\usepackage[normalem]{ulem}
\usepackage{subeqnarray}
\usepackage[framemethod=TikZ]{mdframed}

\usepackage[outdir=./]{epstopdf}

\newcommand{\Bezier}{B\'ezier}

\newcommand{\trunc}{\mathop{\rm trunc}\nolimits}
\newcommand{\overlap}{\mathop{\rm overlap}\nolimits}

\newcommand{\closure}{\mathop{\rm clos}\nolimits}

\newcommand{\argmin}{\mathop{\rm argmin}}

\newcommand{\NN}{{\mathbb N}}

\newcommand{\ZZ}{{\mathbb Z}}
\newcommand{\RR}{{\mathbb R}}

\newcommand{\supp}{\mathop{\rm supp}\nolimits}

\newcommand{\lspan}{\mathop{\rm span}}
\newcommand{\act}{\mathop{\rm in}}
\newcommand{\nact}{\mathop{\rm ex}}

\newcommand{\iset}[3]{\mathcal{I}^{#1}_{#2}(#3)}
\newcommand{\piset}[3]{\widehat{\mathcal{I}}^{#1}_{#2}(#3)}
\newcommand{\eset}[3]{\mathcal{E}^{#1}_{#2}(#3)}
\newcommand{\peset}[3]{\widehat{\mathcal{E}}^{#1}_{#2}(#3)}

\newcommand{\dtrev}[2]{#2} 
\long\def\comment#1{}

\spdefaulttheorem{assumption}{Assumption}{\bf}{}
\numberwithin{lemma}{section}
\numberwithin{proposition}{section}
\numberwithin{theorem}{section}
\numberwithin{corollary}{section}
\numberwithin{definition}{section}
\numberwithin{remark}{section}
\numberwithin{assumption}{section}

\newenvironment{assumpBox}
{
	\begin{mdframed}[
		outerlinewidth=0pt,
		roundcorner=0pt,
		innertopmargin=\baselineskip,
		innerbottommargin=\baselineskip,
		innerrightmargin=14pt,
		innerleftmargin=14pt,
		backgroundcolor=gray!14!white,
		fontcolor=black,
		align=center,
		skipabove=6pt,
		skipbelow=6pt]
	\begin{assumption}
}%
{
	\end{assumption}
	\end{mdframed}
}

\title{A characterization of linear independence of THB-splines in $\RR^n$ and application to \Bezier{} projection}
\titlerunning{\Bezier{} projection for THB-splines}
\author{K. Dijkstra and D. Toshniwal}
\institute{K. Dijkstra \at Delft Institute of Applied Mathematics, Delft University of Technology, Delft, The Netherlands, \email{k.w.dijkstra@tudelft.nl} \and
	D. Toshniwal\at Delft Institute of Applied Mathematics, Delft University of Technology, Delft, The Netherlands, \email{d.toshniwal@tudelft.nl}}
\date{January 2023}

\begin{document}
	
	\maketitle

 
	\begin{abstract}~
		In this paper we propose a local projector for truncated hierarchical B-splines (THB-splines). The local THB-spline projector is an adaptation of the \Bezier{} projector proposed by Thomas et al. (Comput Methods Appl Mech Eng 284, 2015) for B-splines and analysis-suitable T-splines (AS T-splines). For THB-splines, there are elements on which the restrictions of THB-splines are linearly dependent, contrary to B-splines and AS T-splines. Therefore, we cluster certain local mesh elements together such that the THB-splines with support over these clusters are linearly independent, and the \Bezier{} projector is adapted to use these clusters. We introduce general extensions for which optimal convergence is shown theoretically and numerically. In addition, a simple adaptive refinement scheme is introduced and compared to Giust et al. (Comput. Aided Geom. Des. 80, 2020), where we find that our simple approach shows promise.
	\end{abstract}

	
	\section{Introduction}
	In recent years, isogeometric analysis \cite{hughes_isogeometric_2005} has been an active topic of research in numerical mathematics. Using higher regularity finite dimensional spaces for the finite element method allows for better approximation power per degree of freedom \cite{Sande2020,beirao_da_veiga_estimates_2011}. Additionally, this allows the domain to be more accurately imported from Computer Aided Design software, which results in a reduction/elimination of domain meshing errors \cite{hughes_isogeometric_2005, Cottrell2009}. For this reason, B-splines are commonly used as basis functions. However, these B-splines cannot be refined locally. Different types of splines spaces have been introduced that can be locally refined, for example T-splines \cite{sederberg_t-splines_2003}, LR-splines \cite{Dokken2013PolynomialBox-partitions,johannessen_isogeometric_2014}, S-splines \cite{li_s-splines_2019} and HB-splines \cite{forsey_hierarchical_1988, rainer_kraft_adaptive_1998, vuong_hierarchical_2011}. These HB-splines have later been altered in \cite{giannelli_thb-splines_2012} to form the Truncated Hierarchical B-splines (THB-splines), a spline space that has the partition of unity property, and where the basis functions have smaller support than compared to the HB-splines. This short paper introduces a local projector from the standard Sobolev space $L^2(\Omega)$ onto the space of THB-splines.

	Such projectors are of great value for numerical analysis and scientific computing. For example, projectors like these can be used for curve and surface fitting, enforcement of boundary conditions, solution methods with non-conforming meshes, multi-level solver technologies and data compression for image, signal and data processing.
	
	Hierarchical spline fitting has been an active topic of research since the introduction of Hierarchical B-splines. Global fitting methods have been investigated in \cite{greiner_interpolating_1998} while quasi-interpolants have first been extended to hierarchical splines in \cite{rainer_kraft_adaptive_1998}. In \cite{speleers_effortless_2016} and \cite{speleers_hierarchical_2017}, efficient quasi-interpolation has been introduced for THB-splines, where the latter sets up a general framework to construct these from a Sobolev space $W_q^{p+1},1\leq q\leq \infty$. However, these works only explicitly construct interpolates from the continuous space $C(\Omega)$ onto the THB-spline space. More recently, a different local THB-spline projector has been introduced in \cite{giust_local_2020} that requires less function evaluations per degree of freedom then the projector of \cite{speleers_effortless_2016} while being slightly more accurate.
	
	We introduce a \Bezier{} projector for THB-splines. This projector is an extension of the \Bezier{} projector introduced by \cite{Thomas2015} for B-splines and analysis-suitable T-splines. This local projector consists of two steps. For all mesh elements, an initial local $L^2$ projection onto the local polynomial space on that element. Secondly, these $C^{-1}$ projections are smoothed to produce a global spline projection. In this projection, local linear independence of the splines is required for every mesh element. Unfortunately, THB-splines are not linearly independent on each mesh element. In this paper, \dtrev{we introduce assumptions for the mesh, such as grading of the THB-splines, that produce collections of mesh elements. We proof}{under suitable assumptions on the mesh such as graded refinements, we construct local collections of mesh elements, called projection elements, such} that the THB-splines are linearly independent over them. These collections are local in the sense that they consist of adjacent mesh elements, where the number of mesh elements included in a projection element is only dependent on the spline degree.
	To be able to talk about pseudo-local linear independence of functions on collections of mesh elements, we introduce the following definition of overloading.
	\begin{definition}\label{def:overloading}
		For a given space, $\mathbb{S}$, of functions defined on $\Omega$, we say that $\widetilde{\Omega}\subset\Omega$ is overloaded if the functions supported on it are linearly dependent; else, $\widetilde{\Omega}$ is not overloaded.
	\end{definition}
	
	That is, the projection elements that we will introduce will not be overloaded for THB-splines, thus allowing the formulation of a \Bezier{} projection.
	See Figure \ref{fig:MainProjectionStepFigure} for an example application of the projector.
    Various assumptions are required to construct these projection elements, some of which appear in literature, e.g. mesh grading \cite{buffa_adaptive_2016}, while we also introduce two new assumptions. 
    One of these assumptions is non-constructive in the sense that if a mesh violates it, it is unclear which mesh modifications should be performed to satisfy it. Therefore, we also provide a set of stronger and constructive assumptions in two dimensions for quadratic and cubic THB-splines. They are used to formulate a first adaptive refinement scheme.
	
	\begin{figure}[t]
		\centering
		\includegraphics[]{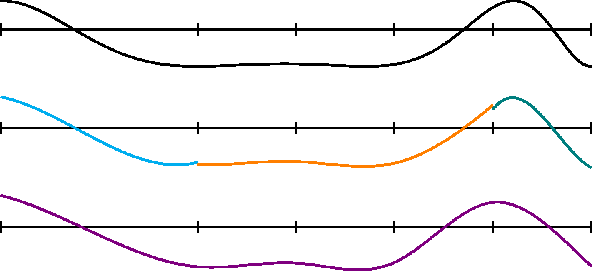}
		\caption{An initial projection on to local collections of mesh elements, and a global smoothing step to produce a global projection of the target onto a THB-spline space.}
		\label{fig:MainProjectionStepFigure}
	\end{figure}
	
	Section \ref{sec:Prelim} recaps the \Bezier{} projector from \cite{Thomas2015} for multivariate B-splines. In Section \ref{sec:THBSplines} we introduce non-overloaded macro-elements that are used in Section \ref{sec:THBProj} to formulate the \Bezier{} projector for THB-splines. In Section \ref{sec:NumRes} the local error estimates are validated, and the adaptive refinement routine is compared to \cite{giust_local_2020}.
	
	\section{B-splines and \Bezier{} projection}
	In this section, B-splines will be briefly introduced and a small recap of \cite{Thomas2015} will be given on \Bezier{} projection for multivariate B-splines.
	For an introductory text on B-splines, we refer the reader to \cite{Lyche2018}.
	
	\label{sec:Prelim}
	\subsection{Univariate B-splines}
	\label{sec:univariate-b-spline}
	Consider the unit domain $\Omega = (0,1)$, a polynomial degree $p\in \mathbb{N}_0$, and an increasing sequence of $m\in \mathbb{N}$ real numbers over $\Omega$:
	\begin{eqnarray}
		\label{eq:knot-sequence}
		\nonumber \boldsymbol{\xi} &:=& \{\xi_1 \leq \xi_2 \leq \dots \leq \xi_m\}\;,\\
		&=& \{\underbrace{0, \dots, 0}_{p+1~\text{times}} < \xi_{p+2} < \dots < \xi_{m-p-1} < \underbrace{1, \dots, 1}_{p+1~\text{times}}\}\;.
	\end{eqnarray}
	This sequence is called a knot sequence, where every individual value is called a knot. This knot sequence can be seen as a partition of $\Omega$, where the unique knots indicate element boundaries for the partition. For our purposes, we will only allow the first and last knots to repeat $p+1$ times, and the other knots are all unique. 
	By convention, we will consider all non-empty open knot spans $(\xi_{k},\xi_{k+1})$ to be the mesh elements.

	For the knot sequence $\boldsymbol{\xi}$, we can define $M := m-p-1$ B-splines recursively,
	\begin{equation}
		B_{j,p,\boldsymbol{\xi}}(x) := \frac{x - \xi_j}{\xi_{j+p} - \xi_j}B_{j,p-1,\boldsymbol{\xi}} + \frac{\xi_{j+p+1}-x}{\xi_{j+p+1}-\xi_{j+1}}B_{j+1,p-1,\boldsymbol{\xi}}\;,~ j=1,\dots,M\;.
	\end{equation}
	Here we assume the fractions to be zero when the denominator is zero. The base case $p=0$ is defined as a unit function over half-open knot intervals:
	\begin{equation}
		B_{j,0,\boldsymbol{\xi}} := \begin{cases} 1& \text{if } x\in [\xi_j,\xi_{j+1})\;, \\ 0 & \text{else}\;.\end{cases}
	\end{equation}
	
	The functions $B_{j,p,\boldsymbol{\xi}}$ are non-negative, form a partition of unity, and the collection of B-splines with support on any given mesh element $\Omega^e = (\xi_k,\xi_{k+1})$ is linearly independent and can reproduce any polynomial of degree $p$ on $\Omega^e$.
	That is, w.r.t. B-splines, none of the mesh elements are overloaded. The span of all $M$ B-splines is called the B-spline space and is denoted as:
	\begin{eqnarray}
		\mathcal{B}_{p,\boldsymbol{\xi}} &:=& \{B_{j,p,\boldsymbol{\xi}}\}_{j=1}^M\;,\\
		\mathbb{B}_{p,\boldsymbol{\xi}} &:=& \lspan \left\{\mathcal{B}_{p,\boldsymbol{\xi}}\right\}\;.
	\end{eqnarray}
	With the spline degree  $p$ fixed, we will omit it from the notation to simply write $B_{j,\boldsymbol{\xi}}$, $\mathcal{B}_{\boldsymbol{\xi}}$ and  $\mathbb{B}_{\boldsymbol{\xi}}$.
	
	\subsection{Multivariate B-splines}
	In the case of a multivariate domain $\Omega = (0,1)^n$, where $n$ denotes the dimension. Let the vector $\vec{p} := (p^1,\dots,p^n)$ denote the polynomial degrees per dimension and let $\Xi = (\boldsymbol{\xi}^1,\dots,\boldsymbol{\xi}^n)$ collect the knot sequences in each dimension.
	In this setting, the multivariate B-spline space $\mathbb{B}_{\vec{p},\Xi}$ is defined as the span of tensor product B-splines over knot sequences $\boldsymbol{\xi}^i$, $i=1,\dots,n$.
	That is, the B-splines $B_{\vec{j},\vec{p},\Xi}$ that span $\mathbb{B}_{\vec{p},\Xi}$ are given by:
	\begin{eqnarray}
		B_{\vec{j},\vec{p},\Xi}(x^1,\dots,x^n) := B_{j^1,p^1,\boldsymbol{\xi}^1}(x^1)\times\dots\times B_{j^n,p^n,\boldsymbol{\xi}^n}(x^n)\;.
	\end{eqnarray}
	Just like in the univariate case, we will assume that the degree $\vec{p}$ is fixed and drop it from all notation to simply denote the B-splines as $B_{\vec{j},\Xi}$, the set of all B-splines as $\mathcal{B}_{\Xi}$, and the spline space as $\mathbb{B}_{\Xi}$.
	
	Here, we also introduce some notation for the mesh corresponding to $\Xi$.
	All $\vec{k}$ such that $p^i+1 \leq k^i \leq m^i - p^i - 1$, define a mesh element $\Omega^e = \bigtimes_{i=1}^n(\xi^i_{k^i},\xi^i_{k^i+1})$. We will index all of these mesh elements in a linear way and store the mesh elements in the set $e\in\mathcal{M}_\Xi$.
	
	\subsection{\Bezier{} projection for B-splines}
	In this section we introduce the \Bezier{} projection for B-splines from \cite{Thomas2015}, which is a local approach to project from $L^2(\Omega)$ onto the space of B-splines, $\mathbb{B}_{\Xi}$.
	There are two main steps in this approach, see Figure \ref{fig:MainProjectionStepBsplines}.
	
	\begin{figure}[t]
        \sidecaption[l]
		\includegraphics[width=0.6\textwidth]{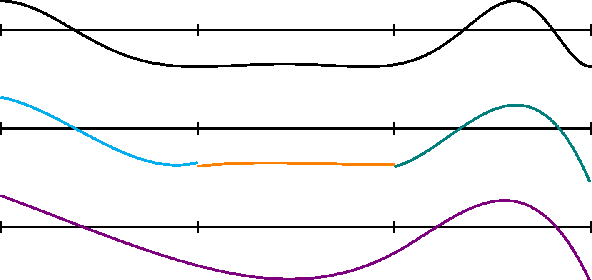}
		\caption{The two main steps of Univariate \Bezier{} B-spline projection. The Target function is initially projected on to a local polynomial basis of degree $p$, on every element. These discontinuous polynomial functions are smoothed to obtain a global projection in the B-spline space $\mathbb{B}_{\boldsymbol{\xi}}$.}
		\label{fig:MainProjectionStepBsplines}
	\end{figure}
	
	Given the target function $f \in L^2(\Omega)$. The first step is to project the target function on to the space of discontinuous polynomials of degree $\vec{p}$.
    In the second step, this piecewise-defined $C^{-1}$ approximation is smoothed out to construct a spline from $\mathbb{B}_{\Xi}$ that approximates $f$.
	
	For the discontinuous polynomial space on element $\Omega^e = \times_{i=1}^n \left(\xi_{k^i}^i,\xi_{k^i+1}^i\right)$, we consider $\boldsymbol{B}_{\Xi^e}$ where $\Xi^e = (\boldsymbol{\xi}^{1,e}, \dots, \boldsymbol{\xi}^{n,e})$ and
	\begin{equation}
		\boldsymbol{\xi}^{i,e} := \big\{
		\underbrace{\xi_{k^i}^i, \dots, \xi_{k^i}^i}_{p^i+1~\text{times}},~\underbrace{\xi_{k^i+1}^i, \dots, \xi_{k^i+1}^i}_{p^i+1~\text{times}}
		\big\},~i=1,\dots,n\;.
	\end{equation}
    This choice is justified as $\mathcal{B}_{\Xi^e}$ spans the polynomial space $\mathbb{P}_{\vec{p}}(\Omega^e)$. The initial $L^2$ projection of $f|_{{\Omega}^e}$ onto the space $\mathbb{P}_{\vec{p}}(\Omega^e)$ gives us
    \begin{eqnarray}\label{eq:UnivariatePolynomialProjection}
        \Pi^e f := \sum_{j^1 = 1}^{p^1+1}\cdots\sum_{j^n = 1}^{p^n+1} a_{\vec{j}}^e B_{\vec{j},\Xi^e}\;,
    \end{eqnarray}
    where $\Pi^e f 
    $ is the solution to the minimization problem $\Vert f - \Pi^e f \Vert_{L^2(\Omega^e)}$.
	Let $\iset{e}{}{\mathcal{B}_{\Xi}}$ denote the set of indices of B-splines supported on $\Omega^e$,
	\begin{equation}
		\iset{e}{}{\mathcal{B}_{\Xi}} := 
		\left\{~
		\vec{j}~:~B_{\vec{j},\Xi}\in \mathcal{B}_{\Xi} ,~\supp(B_{\vec{j},\Xi}) \cap \Omega^e \neq \emptyset~
		\right\}\;,
	\end{equation}
    The target B-spline space $\mathbb{B}_{\Xi}$, restricted to element $\Omega^e$, also spans the local polynomial space $\mathbb{P}_{\vec{p}}(\Omega^e)$. Hence, there exists a unique invertible matrix $\mathrm{C}^e$ such that $\vec{b^e} = \mathrm{C}^e \vec{a}^e$ and:
    \begin{eqnarray}\label{eq:bezier-projection-bspline-form}
		\Pi^e f & =& \sum_{\vec{j} \in \iset{e}{}{\mathcal{B}_{\Xi}}} b_{\vec{j}}^e B_{\vec{j},\Xi}\big|_{\Omega^e}\;.
	\end{eqnarray}
	The matrix $\mathrm{C}^e$ is related to the notions of knot insertion \cite{de_boor_calculating_1972} and \Bezier{} extractions \cite{scott_isogeometric_2011}. In particular, it is the inverse of the \Bezier{} extraction operator for the given element.
	
	Performing the above local projections for each element, we get piecewise-polynomial descriptions of the form \eqref{eq:bezier-projection-bspline-form} on each mesh element.
	However, if $\vec{j} \in \iset{e}{}{\mathcal{B}_{\Xi}} \cap \iset{e'}{}{\mathcal{B}_{\Xi}}$, then in general $b^{e}_{\vec{j}} \neq b^{e'}_{\vec{j}}$. See Figure \ref{fig:MainProjectionStepBsplines}.
	Consequently, the piecewise descriptions together do not define a spline in $\mathbb{B}_{\Xi}$.
	This is rectified by a smoothing operation that, for a given B-spline $B_{\vec{j},\Xi}$, takes all associated coefficients $b^e_{\vec{j}}$ and performs a weighted averaging to yield a single coefficient $b_{\vec{j}}$,
	\begin{eqnarray}
		\label{eq:univariate-b-spline-global-smoothing}
		b_{\vec{j}} &:=& \sum_{e \in \eset{}{\vec{j}}{\mathcal{B}_{\Xi}}} \omega_{\vec{j}}^{e} b^{e}_{\vec{j}}\;,
	\end{eqnarray}
	where $\omega_{\vec{j}}^{e}$ are the averaging weights and $\eset{}{\vec{j}}{\mathcal{B}_{\Xi}}$ is the set of all elements where $B_{\vec{j},\Xi}$ is supported,
	\begin{equation}
		\eset{}{\vec{j}}{\mathcal{B}_{\Xi}} := 
		\left\{
		e \in \mathcal{M}_{\Xi}~:~B_{\vec{j},\Xi}\in\mathcal{B}_{\Xi},~\Omega^e \cap (B_{\vec{j},\Xi})  \neq \emptyset
		\right\}\;.
	\end{equation}
	They can be defined in different ways, we follow the recommendations of \cite{Thomas2015} and choose them as
	\begin{eqnarray}
		\omega_{\vec{j}}^e := \frac{\int_{\Omega^e}B_{\vec{j},\Xi}\;dx}{\int_\Omega B_{\vec{j},\Xi}\;dx}\;.
	\end{eqnarray}
	Observe that $\sum_{e \in \eset{}{\vec{j}}{\mathcal{B}_{\Xi}}} \omega_{\vec{j}}^e = 1$.
	Finally, the global \Bezier{} projection operator $\Pi:L^2(\Omega) \rightarrow \mathbb{B}_{\Xi}$ is defined by
	\begin{equation}
		\Pi f := 
                \sum_{\mathclap{\vec{j}:B_{\vec{j},\Xi}\in\mathcal{B}_{\Xi}}}b_{\vec{j}} B_{\vec{j},\Xi}\;.
	\end{equation}

	\section{Pseudo-local linear independence of THB-splines}\label{sec:THBSplines}
	THB-splines, are a way to build locally refined spline spaces. They are defined in terms of a sequence of nested multivariate B-spline spaces, and a hierarchy of locally refined domains. However, in contrast to the B-spline spaces, certain mesh elements will be overloaded w.r.t. THB-splines (see Definition \ref{def:overloading}). In this section, under suitable assumptions on the mesh, we identify the non-overloaded elements of the mesh and use them to partition the mesh into non-overloaded macro-elements.
	
	\subsection{Construction of THB-splines}
	Consider a nested sequence of multivariate B-spline spaces over domain $\Omega = (0,1)^n$:
	\begin{eqnarray}
		\label{eq:NestedBsplineSpaces}
		\mathbb{B}_{1} \subset \mathbb{B}_{2} \subset \dots \subset \mathbb{B}_{L}\;.
	\end{eqnarray}
	Here the multivariate B-spline space $\mathbb{B}_\ell := \mathbb{B}_{\Xi_\ell} $ containing the B-splines $B_{\vec{j},\ell} := B_{\vec{j}}\in \mathbb{B}_{\Xi_\ell}$, is a level-$\ell$ tensor-product B-spline defined over the knot sequences $\boldsymbol{\xi}^i_\ell$, $i=1,\dots,n$ that define $\Xi_\ell$.
	As before, these knot sequences $\boldsymbol{\xi}^i_\ell$ are all $(p+1)$-open knot sequences, where the interior knots are unique.
	In this setting, nestedness of the B-spline spaces is ensured if and only if:
	\begin{eqnarray}
		\boldsymbol{\xi}_1^i \subset \boldsymbol{\xi}_2^i \subset \dots \subset \boldsymbol{\xi}_L^i\;,\qquad i = 1,\dots,n\;.
	\end{eqnarray}

	We will denote the sets containing level-$\ell$ mesh element indices as $\mathcal{M}_\ell := \mathcal{M}_{\Xi_\ell}$.
	Next, consider a sequence of nested, closed subsets of $\Omega$:
	\begin{eqnarray}
		\Omega_L \subseteq \dots \subseteq \Omega_1 := \overline{\Omega}\;,
	\end{eqnarray}
	where $\Omega_\ell$ is the closure of the union of mesh elements $\Omega^e$ for some $e \in \mathcal{M}_\ell$. The collection of those subsets is denoted by:
	\begin{eqnarray}
		\boldsymbol{\Omega} := \{\Omega_1,\dots,\Omega_L\}\;,
	\end{eqnarray}
	and will be referred to as the domain hierarchy on $\Omega$.
	For a given level $\ell$, we can use this hierarchy to split the B-spline basis functions $\mathcal{B}_\ell$:
	\begin{eqnarray}
		\mathcal{B}_\ell^{\act} &:=& \{B_{\vec{i},\ell} \in \mathcal{B}_\ell ~:~ \supp(B_{\vec{i},\ell}) \subseteq \Omega_\ell\}\;,\\
  \mathcal{B}_\ell^{\nact} &:=& \{B_{\vec{i},\ell} \in \mathcal{B}_\ell ~:~ \supp(B_{\vec{i},\ell}) \nsubseteq \Omega_\ell\}\;.
	\end{eqnarray}
	Using the above, the following defines Hierarchical B-spline basis functions or HB-spline basis functions.
	
	\begin{definition}
		\label{def:HBsplines}
		Given a domain hierarchy $\boldsymbol{\Omega}$, the corresponding set of HB-spline basis functions is denoted by $\mathcal{H}_{\boldsymbol{\Omega}}$ and defined recursively as follows:
		\begin{enumerate}[leftmargin=0.28in]
			\item $\mathcal{H}_1 := \mathcal{B}_1$ ,
			\item for $\ell = 2,\dots,L$ :
			\[\mathcal{H}_\ell := \mathcal{H}_\ell^C \cup \mathcal{H}_\ell^F\;,\]
			where
			\begin{eqnarray*}
				\mathcal{H}_\ell^C &:=& \left\{ B_{\vec{j},k} \in \mathcal{H}_{\ell-1} : \supp(B_{\vec{j},k}) \nsubseteq \Omega_\ell \right\}\;,\\
				\mathcal{H}_\ell^F &:=& \mathcal{B}_\ell^{\act}\;,
			\end{eqnarray*}
			\item $\mathcal{H}_{\boldsymbol{\Omega}} := \mathcal{H}_L$ .
		\end{enumerate}
	\end{definition}

	See the top of plot of Figure \ref{fig:HB-THB-comparison} for an example. HB-spline basis functions lack the partition of unity property, which is shown in red in Figure \ref{fig:HB-THB-comparison}. In addition, the number of overlapping basis functions associated with different hierarchical levels increases easily. This motivates the construction of a different basis, i.e., the THB-spline basis, based on the following truncation mechanism.
	
	\begin{definition}
		Given $\ell=2,\dots,L$. Let $f\in \mathbb{B}_{\ell-1}$ be represented in the B-spline basis $\mathcal{B}_{\ell}$:
		\begin{eqnarray}
			\label{eq:defTruncationSum}
            f = \sum_{\mathclap{\vec{j}:B_{\vec{j},\ell}\in\mathcal{B}_\ell}} c_{\vec{j},\ell}B_{\vec{j},\ell}\;.
		\end{eqnarray}
		The truncation of $f$ at hierarchical level $\ell$ is defined as the the sum of the terms appearing in \eqref{eq:defTruncationSum} corresponding to the B-splines in $\mathcal{B}_\ell^{\nact}$ :
		\begin{eqnarray}
			\label{eq:def-trunc-operator}
			\trunc_\ell(f) := \sum_{\mathclap{\vec{j}:B_{\vec{j},\ell}\in \mathcal{B}_\ell^{\nact}}} {c_{\vec{j},\ell}B_{\vec{j},\ell}}\;.
		\end{eqnarray}
	\end{definition}
	By successively truncating the functions constructed in Definition \ref{def:HBsplines}, THB-spline basis functions are constructed.
	\begin{definition}
		\label{def:THBsplines}
		Given a domain hierarchy $\boldsymbol{\Omega}$, the corresponding set of THB-splines basis is denoted by $\mathcal{T}_{\boldsymbol{\Omega}}$ and defined recursively as follows:
		\begin{enumerate}[leftmargin=0.28in]
			\item $\mathcal{T}_1 := \mathcal{B}_1$ ,
			\item for $\ell = 2,\dots,L$:
			\[\mathcal{T}_\ell := \mathcal{T}_l^C \cup \mathcal{T}_\ell^F\;,\]
			where
			\begin{eqnarray*}
				\mathcal{T}_\ell^C &:=& \left\{ \trunc_\ell(T)~:~T \in \mathcal{T}_{\ell-1}\;,\;\supp(T) \nsubseteq \Omega_\ell \right\}\;,\\
				\mathcal{T}_\ell^F &:=& \mathcal{B}_\ell^{\act}\;,
			\end{eqnarray*}
			\item $\mathcal{T}_{\boldsymbol{\Omega}} := \mathcal{T}_L$ .
		\end{enumerate}
	\end{definition}

	The $N := |\mathcal{T}_{\boldsymbol{\Omega}}|$ THB-spline basis functions are linearly indexed by some ordering $i = 1,\dots,N$:
	\begin{eqnarray}
		\mathcal{T}_{\boldsymbol{\Omega}} = \{T_i\}_{i=1}^N\;.
	\end{eqnarray}
	The space of THB-splines will be denoted as $\mathbb{T}_{\boldsymbol{\Omega}}$.
	Finally, we define the set of active level-$\ell$ mesh elements $\mathcal{M}_\ell^{\act}$ as
	\begin{equation}
		\mathcal{M}_\ell^{\act} :=
		\left\{
		e \in \mathcal{M}_\ell~:~\Omega^e \subset \Omega_\ell\text{~and~}\Omega^e \cap \Omega_{\ell+1}= \emptyset
		\right\}\;.
	\end{equation}
	In Figure \ref{fig:HB-THB-comparison} a comparison between THB-splines and HB-splines is given.

	\begin{figure}
		\centering
		\includegraphics{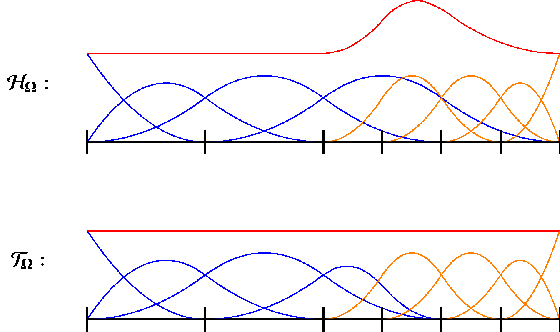}
		\caption{An example of an HB-spline basis (top) and a corresponding THB-spline basis (bottom). The blue splines are from the first level and the orange splines are from the second one. Notice that the total sum of all splines (red line) is 1 for the THB-spline basis.}
		\label{fig:HB-THB-comparison}
	\end{figure}

	\subsection{Some assumption on the hierarchical meshes}
	In order to simplify the developments, we will place the following basic assumptions on our THB-splines for the rest of this document. Additional assumptions will be introduced when necessary.
 
    \begin{assumpBox}
    \label{ass:bisection}
    The mesh sizes at level $\ell = 1$ satisfies quasi-uniformity.
    Moreover, for $\ell > 1$, the level-$\ell$ knot sequence in each direction is obtained by bisecting each non-zero knot span of the corresponding level-$(\ell-1)$ knot sequence.
	\end{assumpBox}
    \begin{assumpBox}
		\label{ass:MeshGrading}
		Given a mesh element $\Omega^e$, $e \in \mathcal{M}^{\act}_\ell$, the only THB-splines that can be supported on $\Omega^e$ correspond to (truncated) B-splines from at most two levels: $\ell-1$ and $\ell$.
	\end{assumpBox}
    
	\subsection{An example of overloaded mesh elements in $\RR^1$}\label{sec:proj-elements}
	
	In contrast to B-splines, mesh elements can be overloaded even for univariate THB-splines. In Figure \ref{fig:overloading-example-1D} an example of an overloaded element is depicted in blue for a quadratic THB-spline space.
	
	\begin{figure}[t]
		\centering
		\includegraphics[]{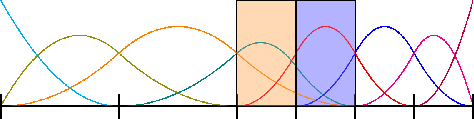}
		\caption{The basis functions of a quadratic THB-spline space consisting of two levels are shown. While the blue coloured element is overloaded, the combination of the blue and orange is not.}
		\label{fig:overloading-example-1D}
	\end{figure}
	
	This univariate example motivates the following observations, which drive the developments in the next sections in the general multivariate setting.
	\begin{itemize}
		\item As a result of the truncation mechanism, at most $p$ elements that are adjacent to the border of $\Omega_\ell$ can be overloaded.
		\item The undesirable overloading on the blue element can be resolved by creating a macro-element that combines the blue and the orange elements.
	\end{itemize}
	
	Then, in the next section, we will characterize the non-overloaded elements for THB-splines and use them to create macro-elements, called projection elements, that are not overloaded.
	
	\subsection{Characterizing non-overloaded mesh elements in $\RR^n$}
	
	\begin{definition}
		Given $\Omega^e = \bigtimes_{i=1}^n \left(\xi_{k^i,\ell}^i~,~ \xi^i_{k^i+1,\ell}\right)$, $e \in \mathcal{M}_\ell$, the following boundary facets,
		\begin{equation}
			\begin{split}
				&\bigtimes_{i=1}^{l-1} \left[\xi_{k^i,\ell}^i~,~ \xi^i_{k^i+1,\ell}\right]
				\bigtimes \left\{\xi_{k^l,\ell}^l\right\}
				\bigtimes_{i=l+1}^{d} \left[\xi_{k^i,\ell}^i~,~ \xi^i_{k^i+1,\ell}\right]\;,\\
				&\bigtimes_{i=1}^{l-1} \left[\xi_{k^i,\ell}^i~,~ \xi^i_{k^i+1,\ell}\right]
				\bigtimes \left\{\xi_{k^l+1,\ell}^l\right\}
				\bigtimes_{i=l+1}^{d} \left[\xi_{k^i,\ell}^i~,~ \xi^i_{k^i+1,\ell}\right]\;,\\
			\end{split}
		\end{equation}
		are called its $l$-normal facets.
	\end{definition}

	\begin{definition}
		Given $\Omega^e = \bigtimes_{i=1}^n \left(\xi_{k^i,\ell}^i~,~ \xi^i_{k^i+1,\ell}\right)$, $e \in \mathcal{M}_\ell$, its $\vec{t} = (t^1,\dots,t^n)$ translation is defined to be the element
		$\tau_{\vec{t}}(\Omega^e) := \bigtimes_{i=1}^n \left(\xi_{k^i+t^i,\ell}^i~,~\xi_{k^i+t^i+1,\ell}^i\right)$.
	\end{definition}

	\begin{definition}
		Given $e \in \mathcal{M}^{\act}_\ell$ and $1 \leq i \leq n$, $d^{i,e}_\ell \in \ZZ_{\geq 0}$ is defined to be the smallest number so that one of the following translations $\vec{t}$ of $\Omega^e$,
		\begin{equation}
			\bigg(
			0\;,\;\dots\;,\;0\;,\;\;\;\mathclap{\underbrace{d^{i,e}_\ell}_{i\text{-th position}}}\;\;,\;0\;,\;\dots\;,\;0
			\bigg)
			~~\text{or}~~
			\bigg(
			0\;,\;\dots\;,\;0\;,\;\;\;\;\mathclap{\underbrace{-d^{i,e}_\ell}_{i\text{-th position}}}\;\;\;,\;0\;,\;\dots\;,\;0
			\bigg)
		\end{equation}
		gives an element $\tau_{\vec{t}}(\Omega^e)$ that has one of its $i$-normal facets contained in $\partial \Omega_\ell \backslash \partial\Omega$.
	\end{definition}

	\begin{definition}\label{def:non-overloaded-border-element}
		For $e \in \mathcal{M}^{\act}_\ell$, $\Omega^e$ is called a well-behaved border element if
		\begin{itemize}[leftmargin=0.28in]
			\item $\partial\Omega^e \cap \left(\partial \Omega_\ell \backslash \partial\Omega\right) \neq \emptyset$, and,
			\item for all $i$, $d^{i,e}_{\ell} = 0$ or $d^{i,e}_{\ell} > p^i$.
		\end{itemize}
	\end{definition}

	For a well-behaved border element $\Omega^e$, let $F^e$ be a set containing some of the $l$-normal facets of $\Omega^e$.
	Let ${F}^{e,\ast}$ denote any subset of $F^e$ be such that the intersection of all facets in ${F}^{e,\ast}$, denoted $\text{intrsct}({F}^{e,\ast}) $, is contained in $\partial \Omega^{e} \cap \left(\partial \Omega_\ell \backslash\partial\Omega \right)$.
	Then, consider the $F^e$ that minimizes the following:
	\begin{eqnarray}\label{eq:def-element-type-boundary}
		\argmin_{F^e}\left\{|F^e|~:~~~\bigcup_{\mathclap{\substack{F^{e,\ast} \subset F^e \\\text{intrsct}(F^{e,\ast})\subset\partial \Omega_\ell \backslash\partial\Omega  } }}\text{intrsct}(F^{e,\ast}) = \partial \Omega^{e} \cap \left(\partial \Omega_\ell \backslash\partial\Omega \right) \right\}\;.
	\end{eqnarray}
	Then, we associate a direction to $\Omega^e$, denoted $\vec{n}^e_\ell$, and define it as the sum of unit normals for each facet  in $F^e$ (by convention, we assume that each unit normal is pointing into $\Omega^e$).
	
	\begin{assumpBox}\label{ass:wide-refinements}
		For any well-behaved border element  $\Omega^e =$
  \newline \noindent$ \bigtimes_{i=1}^n \left(\xi_{k^i,\ell}^i~,~ \xi^i_{k^i+1,\ell}\right) $ and the direction $\vec{n}^{e}_\ell$ associated to it,
		the following containment must hold,
		\begin{equation}\label{eq:wide-refinement-domain}
			\bigtimes_{i=1}^n \left(\xi_{k^i + (n^{i,e}_\ell-1) p^i ,\ell}^i~,~ \xi^i_{k^i+(n^{i,e}_\ell+1) p^i + (1-n^{i,e}_\ell),\ell}\right) \subseteq \Omega_\ell\;.
		\end{equation}
	\end{assumpBox}

	\begin{definition}\label{def:projection-element-boundary}
		Given a well-behaved border element $\Omega^e = \bigtimes_{i=1}^n \left(\xi_{k^i,\ell}^i~,~ \xi^i_{k^i+1,\ell}\right)$, and the direction $\vec{n}^{e}_\ell$ associated to it,
		we define the projection element generated by $\Omega^e$ as
		\begin{equation}
			\widehat{\Omega}^e = \bigtimes_{i=1}^n \left(\xi_{k^i,\ell}^i~,~ \xi^i_{k^i+1+n^{i,e}_\ell (p^i-1),\ell}\right)\;.
		\end{equation}
	\end{definition}
	
     \begin{lemma}\label{lem:non-decreasing-trunc}
		Let $\Omega^e$ be a well-behaved border element, and let $\tau_{\vec{t}_1}(\Omega), \tau_{\vec{t}}(\Omega) \subset \widehat{\Omega}^e $, $|t_1^i|\geq |t^i|$ for all $i$, such that they are both contained in the same level $\ell-1$ element. If $T_j$ is the truncation of some level $\ell-1$ B-spline and does not vanish on $\tau_{\vec{t}_1}(\Omega)$, then $T_j$ doesn't vanish on $\tau_{\vec{t}}(\Omega)$.
	\end{lemma}
    \begin{proof}
		Assume without loss of generality that all components of $\vec{n}^{e}_\ell$ are non-negative.
		This implies that $n^{i,e}_\ell (p^i - 1) \geq t^i_1 \geq t^i \geq 0$ for all $i$.
		
		Let $T_j$ be equal to the truncation of the B-spline $B_{\vec{r},\ell-1}$.
        Note that $B_{\vec{r},\ell-1}$ has support on both elements $\tau_{\vec{t}_1}(\Omega)$ and $\tau_{\vec{t}}(\Omega)$.
		Then, for some $s^1,\dots,s^n$, we can express $B_{\vec{r},\ell-1}$ as
		\begin{equation}
			B_{\vec{r},\ell-1}
			=
			\sum_{q^1=s^1}^{s^1+p^1+1}\cdots\sum_{q^n=s^n}^{s^n+p^n+1}
			c_{\vec{q}}B_{\vec{q},\ell}\;.
		\end{equation}
		Since $T_j$ has support on $\tau_{\vec{t}_1}(\Omega)$, there is a non-zero $c_{\vec{q}}$ for $B_{\vec{q},\ell}\in\mathcal{B}_\ell^{\nact}$ with support on $\tau_{\vec{t}_1}(\Omega^e)$. If $B_{\vec{q},\ell}$ has support on $\tau_{\vec{t}}(\Omega^e)$, we are done. Otherwise,
		\begin{equation}
            q^i \geq s^i + t^i_1 - t^i,~ i = 1, \dots, n\;.
		\end{equation}
        As $B_{\vec{q},\ell}\in\mathcal{B}_\ell^{\nact}$ has support on some element $e'\in\mathcal{M}_{\ell-1}^{\act}$ and since $0\leq t_1^i-t^i \leq 1$, we have that $B_{\vec{q}-\vec{t}_1+\vec{t},\ell}$ has support on the same level $\ell-1$ element $\Omega^{e'}$ and thus $B_{\vec{q}-\vec{t}_1+\vec{t},\ell}\in\mathcal{B}_\ell^{\nact}$. But then, $B_{\vec{q}-\vec{t}_1+\vec{t},\ell}$ and thus $T_j$ have support on $\tau_{\vec{t}}(\Omega^e)$.
	\end{proof}
	
	\begin{assumpBox}\label{ass:unique-projection-elements}
		Let $\Omega^e$, $e \in \mathcal{M}_\ell^{\act}$, be so that its $\vec{t} = (t^1,\dots,t^n)$ translation, where $|t^i| < p^i$ for all $i$, satisfies
		$\partial\tau_{\vec{t}}(\Omega^{e}) \cap \left(\partial \Omega_\ell \backslash\partial\Omega\right)\neq \emptyset$.
		Then, there exists a unique projection element that contains $\Omega^e$.
	\end{assumpBox}

	
	\begin{proposition}\label{prop:lin-ind-propagation}
		Consider a well-behaved border element $\Omega^e$ and consider $\tau_{\vec{t}_1}(\Omega^e) \subset \widehat{\Omega}^e$ for some $\vec{t}_1 = (t_1^1,\dots,t_1^n)$.
		Then, the following set of THB-splines on $\tau_{\vec{t}_1}(\Omega^{e})$ is linearly independent,
		\begin{equation}\label{eq:prop-lin-ind-propagation-THB-splines}
			\begin{split}
				\bigg\{
				T_j:
				&\supp(T_j) \cap \tau_{\vec{t}_1}(\Omega^{e}) \neq \emptyset
				\text{,~and,~}\\
				&\supp(T_j) \cap \tau_{\vec{t}}(\Omega^{e}) = \emptyset
				\text{~where~}
				\tau_{\vec{t}}(\Omega^{e}) \subset \widehat{\Omega}^e\;,\;
				\vec{t}_1 \neq \vec{t}
				\text{~and~}
				|t_1^i| \geq |t^i|\text{~for all~}i
				\bigg\}\;.
			\end{split}
		\end{equation}
	\end{proposition}
	\begin{proof}
		Assume without loss of generality that $\vec{n}^{e}_\ell$ is such that all of its components are non-negative, and let $\Omega^e = \bigtimes_{i=1}^n \left(\xi_{k^i,\ell}^i~,~ \xi^i_{k^i+1,\ell}\right)$.
		In the following, we use $\vec{t}$ to denote a translate vector such that $\tau_{\vec{t}}(\Omega^{e}) \subset \widehat{\Omega}^e$, $\vec{t}_1 \neq \vec{t}$,
		$t_1^i \geq t^i$ for all $i$.
		
		The claim is immediate for the special case when there exists a $\vec{t}$ such that $\tau_{\vec{t}_1}(\Omega^e)$ and $\tau_{\vec{t}}(\Omega^e)$ are contained in a common element $e'\in\mathcal{M}_{\ell-1}$.
		In this case, from Lemma \ref{lem:non-decreasing-trunc}, the THB-splines that contain $\tau_{\vec{t}_1}(\Omega^e)$ in their support but not $\tau_{\vec{t}}(\Omega^e)$ only correspond to level-$\ell$ B-splines.
		The claim then follows from the linear independence of level-$\ell$ B-splines.
		
		The only other case to consider is when $\tau_{\vec{t}_1}(\Omega^e)$ does does not share an ancestor for any  $\tau_{\vec{t}}(\Omega^e)$ satisfying the above conditions.
		In particular, this means that all $t^i_1$ are even numbers. Here, we will show that the desired set of THB-splines \eqref{eq:prop-lin-ind-propagation-THB-splines} spans a subspace of $\mathbb{P}_{\vec{p}}\left(\tau_{\vec{t}_1}(\Omega^e)\right)$, where each polynomial in the subspace vanishes $p^i$ times at the hyperplane $\xi^i = \xi^i_{k^i+t_1^i,\ell}$ if $t_1^{i} > 0$.
		
	    Let the element $e'\in\mathcal{M}_{\ell-1}$	contain $\Omega^e \supset \Omega^{e'}$ and let the level-$(\ell-1)$ B-splines supported on $\Omega^{e'}$, be $B_{\vec{q},\ell-1}$, where
		\begin{equation}
			\vec{q} \in \iset{e'}{}{\mathcal{B}_{\ell-1}}= \bigtimes_{i=1}^n
			\left\{
			r^i\;,\;\dots\;,\;r^i + p^i
			\right\}\;.
		\end{equation}
		Then, the B-splines supported on $\tau_{\vec{t}_1}(\Omega^e)$ and not on any $\tau_{\vec{t}}(\Omega^e)$ as above correspond to $\vec{q} \in \bigtimes_{i=1}^n S^{i}_{\ell-1}$:
		\begin{eqnarray}
			S^{i}_{\ell-1} := \begin{cases}
				\left\{r^i+p^i+\frac{t^i_1}{2} \right\}\;,& t^{i}_1 > 0\;,\\
				\left\{r^i,\dots,r^i+p^i \right\}\;,& \text{else}\;.
			\end{cases}
		\end{eqnarray}
		Similarly, let the level-$\ell$ B-splines supported on $\Omega^{e}$ be $B_{\vec{q},\ell}$, where
		\begin{equation}
			\vec{q} \in \iset{e}{}{\mathcal{B}_\ell} = \bigtimes_{i=1}^n
			\left\{
			s^i\;,\;\dots,\;s^i + p^i
			\right\}\;.
		\end{equation}
		Then, the B-spline supported on $\tau_{\vec{t}_1}(\Omega^e)$ and not $\tau_{\vec{t}}(\Omega^e)$ correspond to $\vec{q} \in \bigtimes_{i=1}^n S^{i}_{\ell}$ where
		\begin{eqnarray}
			S^{i}_{\ell} := \begin{cases}
				\left\{s^i+p^i+t^i_1 \right\}\;,& t^{i}_1 > 0\;,\\
				\left\{s^i,\dots,s^i+p^i \right\}\;,& \text{else}\;.
			\end{cases}
		\end{eqnarray}
		Moreover, the active level-$\ell$ B-splines from the above set correspond to $B_{\vec{q},\ell}$, $\vec{q} \in \bigcup_{l=1}^n\bigtimes_{i=1}^n A^{i,l}_\ell$, 
		where
		\begin{equation}
			A^{i,l}_\ell :=
			\begin{dcases}
				\left\{s^i + p^i+t^i_1\right\}\;, & t^i_1 > 0\;,\\
				\left\{
				s^i + p^i
				\right\}\;, & t^i_1 = 0\text{~and~}d^{i,e}_\ell = 0\text{~or~}i = l\;,\\
				\left\{
				s^i\;,\;\dots\;,\;s^i + p^i
				\right\}\;, & t^i_1 = 0\text{~and~}d^{i,e}_\ell > p^i\text{~and~}i \neq l\;.
			\end{dcases}
		\end{equation}
		
		The level-$(\ell-1)$ B-splines from $\bigtimes_{i=1}^nS^i_{\ell-1}$ that are truncated to 0 on $\tau_{\vec{t}_1}(\Omega^e)$ by the above active level-$\ell$ B-splines are $B_{\vec{q},\ell-1}$, $\vec{q} \in \bigcup_{l=1}^n\bigtimes_{i=1}^n I^{i,l}_{\ell-1}$, where
		\begin{equation}
			I^{i,l}_{\ell-1} :=
			\begin{dcases}
				\left\{
				r^i + p^i + \frac{t^i_1}{2}
				\right\}\;,& t^i_1 > 0\;,\\
				\left\{
				r^i + p^i
				\right\}\;, & t^i_1 = 0 \text{~and~}d^{i,e}_\ell = 0\text{~or~}i = l\;,\\
				\left\{
				r^i\;,\;\dots\;,\;r^i + p^i
				\right\}\;, & t^i_1 = 0 \text{~and~}d^{i,e}_\ell > p^i \text{~and~}i = l\;,\\
			\end{dcases}
		\end{equation}
		
		From the above, we notice that the cardinalities of $\bigtimes_{i=1}^n S^{i}_{\ell-1}$ and $\bigtimes_{i=1}^n S^{i}_{\ell}$ are the same, and the number of level-$(\ell-1)$ THB-splines truncated to zero above is exactly equal to the number of active level-$\ell$ THB-splines.
		Consequently, the complete set of THB-splines supported on $\tau_{\vec{t}_1}(\Omega^e)$ and not $\tau_{\vec{t}}(\Omega^e)$ is linearly independent, and each THB-spline corresponds to a truncated level-$(\ell-1)$ B-splines from the following set,
		\begin{equation}
			\bigtimes_{i=1}^n S^{i}_{\ell-1} \bigg\backslash \bigcup_{l=1}^n\bigtimes_{i=1}^n I^{i,l}_{\ell-1}\;,
		\end{equation}
		or a level-$\ell$ B-spline from the following set,
		\begin{equation}
			\bigcup_{l=1}^n\bigtimes_{i=1}^n A^{i,l}_\ell\;.
		\end{equation}
	\end{proof}

	\begin{corollary}\label{cor:good-border-elements-non-overloaded}
		Well-behaved border elements are not overloaded.
	\end{corollary}
	\begin{proof}
		The result follows by considering $\vec{t} = (0,\dots,0)$ in Proposition \ref{prop:lin-ind-propagation}.
	\end{proof}
	
	\begin{theorem}\label{thm:non-overloaded-projection-element}
		Let $\Omega^e$ be a well-behaved border element.
		Then, $\widehat{\Omega}^e$ is not overloaded.
	\end{theorem}
	\begin{proof}
		Consider the following representation of the zero function on $\widehat{\Omega}^e$ as a linear combination of THB-splines on $\widehat{\Omega}^e$,
		\begin{equation}
			0 = \sum_{j}c_jT_j|_{\widehat{\Omega}^e}\;.
		\end{equation}
		Starting from the linear independence on the well-behaved border element $\Omega^e$ shown in Corollary \ref{cor:good-border-elements-non-overloaded}, we can conclude that some of the $c_j$ in the above sum should be zero.
		Consequently, repeated applications of Proposition \ref{prop:lin-ind-propagation} can be used to show that all the remaining $c_j$ should be zero as well since every element in $\widehat{\Omega}^e$ can be obtained by translations of $\Omega^e$.
	\end{proof}

	We finish this section by defining well-behaved interior elements and the corresponding projection elements.
	
	\begin{definition}\label{def:projection-element-interior}
		Given $e \in \mathcal{M}^{\act}_\ell$, $\Omega^e$ is called a well-behaved interior element if the THB-splines supported on it only correspond to level-$\ell$ B-splines.
		The corresponding projection element is defined as $\widehat{\Omega}^e := \Omega^e$.
	\end{definition}

	Clearly, a well-behaved interior element is not overloaded.
	Moreover, as a result of Assumption \ref{ass:unique-projection-elements}, each mesh element is contained in a unique projection element generated by a well-behaved (interior or boundary) element.	

 \section{\Bezier{} projection for THB-splines}
\label{sec:THBProj}
Based on the main result of the previous section, we formulate a \Bezier{} projection for THB-splines by using the notion of projection elements. There are various ways to extend the \Bezier{} projector from \cite{Thomas2015}, and we will show theoretical error estimates for these extensions.

\subsection{General formulation of the projector}
From the results and assumptions from Section \ref{sec:THBSplines}, we can make two observations:
\begin{itemize}
	\item each mesh element is contained in a unique projection element $\widehat{\Omega}^e$ (c.f. Definitions \ref{def:projection-element-boundary} and \ref{def:projection-element-interior} and Assumptions \ref{ass:wide-refinements} and \ref{ass:unique-projection-elements}), and,
	\item none of the projection elements are overloaded (c.f. Theorem \ref{thm:non-overloaded-projection-element}).
\end{itemize}
As a consequence, we can formulate the following generalized \Bezier{} projection methodology for THB-splines.

\paragraph{\textbf{Step 1: Local projections on $\widehat{\Omega}^e$}}
Given an arbitrary well-behaved (boundary or interior) element $\Omega^e$, $e \in \mathcal{M}^{\act}_\ell$, denote with $\mathbb{V}(\widehat{\Omega}^e)$ some local spline space defined on $\widehat{\Omega}^e$; we will make the choice of this spline space concrete in the next subsection.
Then, in the first step of the projection, and given a target function $f \in L^2(\Omega)$, we project $f|_{\widehat{\Omega}^e}$ onto the space $\mathbb{V}(\widehat{\Omega}^e)$. Denote this projection operator as $\widehat{\Pi}^e_0$. Let $\widehat{V}_j$, $j = 1, \dots, \widehat{M}^e := \dim(\mathbb{V}(\widehat{\Omega}^e))$, form a basis of $\mathbb{V}(\widehat{\Omega}^e)$, such that we can write
\begin{equation}\label{eq:local-projection-prelim-thb}
	\widehat{\Pi}^e_0 f = \sum_{j = 1}^{\widehat{M}^e} \hat{f}_j^e \widehat{V}_j\;.
\end{equation}

In the next step, we apply a subsequent $L^2$-projection $\widehat{\Pi}^e_1~:~\mathbb{V}(\widehat{\Omega}^e) \rightarrow \widehat{\mathbb{T}}^e_{\boldsymbol{\Omega}}$, where $\widehat{\mathbb{T}}^e_{\boldsymbol{\Omega}}$ is defined to be the restriction of the THB-spline space  $\mathbb{T}_{\boldsymbol{\Omega}}$ to the projection element $\widehat{\Omega}^e$.
Let $\piset{e}{}{\mathcal{T}_{\boldsymbol{\Omega}}}$ contain the indices of THB-spline basis functions supported on $\widehat{\Omega}^e$.
These are linearly independent by Theorem \ref{thm:non-overloaded-projection-element}.
The action of $\widehat{\Pi}^e_1$ can be encoded in a matrix $\widehat{\mathrm{D}}^e$ such that $\widehat{\vec{c}}^e=\widehat{\mathrm{D}}^e \widehat{\vec{f}}^e$ and:
\begin{equation}
	\label{eq:local-projection-final-thb}
	\widehat{\Pi}^e_1\widehat{\Pi}^e_0 f = 
    \sum_{\mathclap{j \in \piset{e}{}{\mathcal{T}_{\boldsymbol{\Omega}}}}} \hat{c}^e_jT_j|_{\widehat{\Omega}^e} \;.
\end{equation}

\paragraph{\textbf{Step 2: Weighted-averaging to form the global projection on $\Omega$}}
For a THB-spline $T_j$, the above process gives a coefficient $\hat{c}^e_j$ for each projection element $\widehat{\Omega}^e$ on which $T_j$ does not vanish.
Then, as in the B-spline case, we perform a weighted averaging of the coefficients $\hat{c}^e_j$ to define a unique coefficient for $T_j$,
\begin{eqnarray}
	\label{eq:thb-spline-global-smoothing}
	c_{j} &:=& \sum_{e \in \peset{}{j}{\mathcal{T}_{\boldsymbol{\Omega}}}} \widehat{\omega}_{j}^{e} \hat{c}^{e}_{j}\;,
\end{eqnarray}
where $\widehat{\omega}_{j}^{e}$ are the averaging weights and $\peset{}{j}{\mathcal{T}_{\boldsymbol{\Omega}}}$ is the set of all projection elements where $T_{j}$ does not vanish,
\begin{equation}
	\peset{}{j}{\mathcal{T}_{\boldsymbol{\Omega}}} := 
	\left\{
	e~:~T_j\in\mathcal{T}_{\boldsymbol{\Omega}}~,~\widehat{\Omega}^e \cap \supp(T_{j}) \neq \emptyset
	\right\}\;.
\end{equation}
In particular, similarly to B-splines, we choose the weights to be,
\begin{equation}
	\widehat{\omega}_{j}^{e} := \frac{\int_{\widehat{\Omega}^e}T_{j}\;dx}{\int_\Omega T_j\;dx}\;.
\end{equation}
This leads to the following definition of the \Bezier{} projector for THB-splines, denote $\Pi~:~L^2(\Omega) \rightarrow \mathbb{T}_{\boldsymbol{\Omega}}$,
\begin{equation}
	\Pi f := \sum_{j=1}^N c_j T_j\;.
\end{equation}


\subsection{Possible choices for $\mathbb{V}(\widehat{\Omega}^e)$}\label{sec:local-space-choices}
The above general methodology provides several options for formulating the \Bezier{} projection operator for THB-splines by varying the choice of $\mathbb{V}(\widehat{\Omega}^e)$ on each projection element $\widehat{\Omega}^e$.
We outline three such choices over here.
Two common features of the following choices are that:
\begin{itemize}
	\item they lead to local projections $\widehat{\Pi}^e_1\widehat{\Pi}^e_0$ that yield optimal a priori error estimates with respect to the mesh size, and,
	\item the corresponding matrices $\widehat{\mathrm{D}}^e$ are independent of the mesh size, in the sense that they stay invariant when the domain and the hierarchical mesh are both scaled up or down.
\end{itemize}

\paragraph{\textbf{Choice 1: global polynomials on $\widehat{\Omega}^e$}}
As a first choice, one can pick $\mathbb{V}(\widehat{\Omega}^e) = \mathbb{P}_{\vec{p}}(\widehat{\Omega}^e)$.
With this choice, the projection $\widehat{\Pi}^e_0$ can be formulated exactly as in \eqref{eq:UnivariatePolynomialProjection}, albeit for a projection element $\widehat{\Omega}^e$ instead of a mesh element $\Omega^e$.
Moreover, since $\mathbb{P}_{\vec{p}}(\widehat{\Omega}^e) \subset \widehat{\mathbb{T}}^e_{\boldsymbol{\Omega}}$, the projection $\widehat{\Pi}^e_1$ is the identity and knot insertion can be used to arrive at the corresponding matrix $\widehat{\mathrm{D}}^e$.
However, \underline{we do not consider this choice} because the resulting $\Pi$ does not preserve THB-splines, and hence is not a projector.

\paragraph{\textbf{Choice 2: THB-splines on $\widehat{\Omega}^e$}}
An alternative is to directly project onto the local THB-spline space by choosing $\mathbb{V}(\widehat{\Omega}^e) = \widehat{\mathbb{T}}^e_{\boldsymbol{\Omega}}$.
Here, the projection $\widehat{\Pi}^e_0$ can be formulated as in \eqref{eq:UnivariatePolynomialProjection} but then using the basis functions $T_j|_{\widehat{\Omega}^e}$.
Consequently, the projection $\widehat{\Pi}^e_1$ is again the identity and the corresponding matrix $\widehat{\mathrm{D}}^e$ is also an identity matrix.

\paragraph{\textbf{Choice 3: discontinuous piecewise-polynomials on $\widehat{\Omega}^e$}}
The final choice we propose is to pick $\mathbb{V}(\widehat{\Omega}^e) = \mathbb{B}^{-1}_{\vec{p},\widehat{\Omega}^e}$, where $\mathbb{B}^{-1}_{\vec{p},\widehat{\Omega}^e}$ is defined to be the space of discontinuous piecewise-polynomials on $\widehat{\Omega}^e$.
With this choice, the projection $\widehat{\Pi}^e_0$ can be formulated as a combination of $L^2$-projections onto $\mathbb{P}_{\vec{p}}(\Omega^{e'})$, see \eqref{eq:UnivariatePolynomialProjection}, for every $\Omega^{e'} \subset \widehat{\Omega}^e$.
In general, $\mathbb{V}(\widehat{\Omega}^e) \supset \widehat{\mathbb{T}}^e_{\boldsymbol{\Omega}}$, and so the projection $\widehat{\Pi}^e_1$ is equivalent to a least-squares problem and the matrix $\widehat{\mathrm{D}}^e$ is the pseudo-inverse of the corresponding linear system.

\subsection{A priori local error estimates}
We begin this section by defining the support extensions of elements, and showing that the assumptions placed so far imply that the size of the support extension is bounded by the local mesh size.

\begin{definition}
    We define the support extension of $\Omega^e$, $e \in \mathcal{M}^{\act}_\ell$, w.r.t. $\mathcal{T}_{\boldsymbol{\Omega}}$ as:
    \begin{eqnarray}
        \widetilde{\Omega}^e := \bigcup_{
        	\mathclap{
        	\substack{
        		j \in \iset{e}{}{\mathcal{T}_{\boldsymbol{\Omega}}}\\
                e' \in \peset{}{j}{\mathcal{T}_{\boldsymbol{\Omega}}}
        		}
        	}
        }
         \closure\left(\widehat{\Omega}^{e'}\right)\;,
    \end{eqnarray}
\end{definition}

Define the mesh size associated to an element $\Omega^e = \bigtimes_{i=1}^n \left(\xi_{k^i,\ell}^i~,~ \xi^i_{k^i+1,\ell}\right)$, $e \in \mathcal{M}^{\act}_\ell$, as
\begin{equation}
	h^e := \max_i \left(\xi^i_{k^i+1,\ell} - \xi_{k^i,\ell}^i\right)\;.
\end{equation}
From Assumptions \ref{ass:bisection} and \ref{ass:MeshGrading} placed on the hierarchical mesh, and by definition of the projection elements, there exists a constant $C_{0}$, dependent only on the degree, such that the mesh size of a projection element is
\begin{equation}
	h^e \leq \widehat{h}^e \leq C_{0}h^e\;.
\end{equation}
Let the corresponding support extension $\widetilde{\Omega}^e$ be contained in the smallest bounding box $\bigtimes_{i=1}^n \left(\widetilde{\xi}^i_0~,~\widetilde{\xi}^i_1\right)$.
Then, the mesh size for $\widetilde{\Omega}^e$ is defined as
\begin{equation}
	\widetilde{h}^e := \max_i \left(\widetilde{\xi}^i_{1} - \widetilde{\xi}_{0}^i\right)\;.
\end{equation}
Again, there is a constant, $C_{1}$, independent of $h^e$, such that
\begin{equation}
	h^e \leq \widetilde{h}^e \leq C_{1} h^e\;.
\end{equation}
Using the above, we now state the following local error estimates for the \Bezier{} projector.
\begin{theorem}
\label{thm:LocalTHBsplineProjEst}
    For $e\in \mathcal{M}_\ell^{\act}$, $0\leq k \leq m \leq \min(\vec{p})+1$ and $f\in H^m(\widetilde{\Omega}^{e})$:
    \begin{eqnarray}
        \vert f - \Pi f \vert_{H^k({\Omega}^{e})} \leq C (h^e)^{m-k}\vert f \vert_{H^m(\widetilde{\Omega}^{e})}\;,
    \end{eqnarray}
	where $C$ is a constant independent of the mesh size $h^e$.
\end{theorem}

The proof for this theorem is an adaptation of the proof given in \cite{Thomas2015} and relies on standard polynomial approximation estimates, and spline-reproduction and stability of the projection $\Pi$.

\begin{lemma}
\label{lem:LocStabPresSpline}
    The projector $\Pi$ has the following properties:
    \begin{itemize}[leftmargin=0.28in]
        \item $\Pi f = f$ for $f\in \mathbb{T}_{\boldsymbol{\Omega}}$, \hfill \textit{(spline reproduction)}
        \item $\Vert\Pi f\Vert_{L^2(\Omega^e)} \leq C_{\text{stab}}\Vert f\Vert_{L^2(\widetilde{\Omega}^{e})}$ for $f\in L^2(\widetilde{\Omega}^{e})$,\hfill \textit{(local stability)}
    \end{itemize}
    where $C_{\text{stab}}$ is independent of the mesh size $h^e$.
\end{lemma}
\begin{proof}
    Spline preservation follows trivially by construction since every step (and with Choices 2 and 3 for the local spaces) perfectly preserves the piecewise-polynomial representation of the THB-spline.\footnote{Choice 1 will instead preserve the polynomials of $\mathbb{\Omega}$.} For local stability, observe that on $\Omega^e$:
    \begin{eqnarray*}
        \Pi f|_{\Omega^e} = \sum_{j\in \iset{e}{}{\mathcal{T}_{\boldsymbol{\Omega}}}}\left(\sum_{e \in \peset{}{j}{\mathcal{T}_{\boldsymbol{\Omega}}}}\widehat{\omega}_j^e \hat{c}^e_j \right)T_j\;.
    \end{eqnarray*}
	Then,
    \begin{equation}
    	\begin{split}
    		\bigg\lvert\Pi f|_{\Omega^e}\bigg\rvert
    		\leq \max_{j\in \iset{e}{}{\mathcal{T}_{\boldsymbol{\Omega}}}}\left\lvert\sum_{e' \in \peset{}{j}{\mathcal{T}_{\boldsymbol{\Omega}}}}\widehat{\omega}_j^{e'} \hat{c}^{e'}_j \right\rvert \sum_{j\in \iset{e}{}{\mathcal{T}_{\boldsymbol{\Omega}}}}T_j
    		\leq \max_{
    			\substack{
    				j \in \iset{e}{}{\mathcal{T}_{\boldsymbol{\Omega}}}\\
    				e' \in \peset{}{j}{\mathcal{T}_{\boldsymbol{\Omega}}}
    			}}
    			\left\lvert\hat{c}^{e'}_j \right\rvert\;,
    	\end{split}
    \end{equation}
	where we have used partition-of-unity of THB-splines.
    For any given $e$, $\hat{c}_j^e$ are calculated using \eqref{eq:local-projection-final-thb} and, since the matrix $\widehat{\mathrm{D}}^e$ is mesh size independent, $\| \widehat{\mathrm{D}}^e\|_\infty$ is mesh-size independent. Moreover, the values $\hat{f}_j^e$ in \eqref{eq:local-projection-prelim-thb} computed via an $L^2$-projection satisfy
    \begin{eqnarray}
        \max_{k}|\hat{f}_k^e| 
        \leq \frac{C}{\sqrt{\left(\hat{h}^e\right)^n}}\Vert f \Vert_{L^2(\widehat{\Omega}^e)}
        \leq \frac{C}{\sqrt{(h^e)^n}}\Vert f \Vert_{L^2(\widehat{\Omega}^e)}\;.
   \end{eqnarray}
   Combining the above and integrating over $\Omega^e$:
    \begin{eqnarray*}
        \Vert \Pi f \Vert_{L^2(\Omega^{e})}
        \leq C_{\text{stab}} \Vert f \Vert_{L^2(\widetilde{\Omega}^e)}\;,
    \end{eqnarray*}
	for a constant $C_{\text{stab}}$ that is independent of $h^e$.
\end{proof}

\begin{proof}[Theorem \ref{thm:LocalTHBsplineProjEst}]
    For any polynomial $f_h \in \mathbb{P}_{\vec{p}}(\widetilde{\Omega}^e)$,
    \begin{eqnarray*}
        \vert f - \Pi f \vert_{H^k(\Omega^e)} &=& \vert f - f_h + f_h - \Pi f \vert_{H^k(\Omega^e)}\\
        &\leq& \vert f - f_h \vert_{H^k(\Omega^e)} + \vert \Pi (f-f_h) \vert_{H^k(\Omega^e)}\;,\\
        &\leq& \vert f - f_h \vert_{H^k(\widetilde{\Omega}^e)} + C_{\text{inv}}C_{\text{stab}}(h^e)^{-k}\Vert f-f_h \Vert_{L^2(\widetilde{\Omega}^e)}\;,
    \end{eqnarray*}
    where we have used the inverse inequality for polynomials and local stability of $\Pi$ from Lemma \ref{lem:LocStabPresSpline}. Consequently, using standard polynomial approximation estimates for both terms, we obtain the desired claim for a constant $C$ that depends on $C_{\text{pol}}$, $C_{\text{stab}}$, $C_{\text{inv}}$ and the polynomial degree.
\end{proof}

\section{Numerical results}
\label{sec:NumRes}
We perform several numerical experiments to validate our findings. Here we limit ourselves to the two dimensional THB-spline spaces of various polynomial degrees. First, the local estimate from Theorem \ref{thm:LocalTHBsplineProjEst} is numerically validated.
Next, we perform tests to compare the proposed \Bezier{} projection to \cite{giust_local_2020}.
Finally, for all numerical tests, we adopt Choice 3 from Section \ref{sec:local-space-choices}.

\subsection{Verification of error estimates}
To validate the error estimates of Theorem \ref{thm:LocalTHBsplineProjEst}, the target function $f(x,y) = \sin(\pi x)\sin(\pi y)$ is projected onto $\mathbb{T}_{\boldsymbol{\Omega}}$ with $L = 2$ levels, where
\begin{equation}\label{eq:local-estimates-numerical-problem}
	\Omega_1 := \overline{\Omega}\;,\qquad \Omega_2 := \{(x,y)\in \Omega_1:x\geq 0.5,y\geq 0.5\}.
\end{equation}
On the above domain hierarchy, THB-spline spaces $\mathbb{T}_{\boldsymbol{\Omega}}$ are constructed for each $\vec{p} \in \{(p,p)~:~p = 1, \dots, 5 \}$. The coarsest mesh consists of $2\times 2$ elements, and the subsequent refinements are built by bisecting the meshes at all levels.
The results can be seen in Figure \ref{fig:NumResConvergence}, the error converges optimally w.r.t. the mesh size.
We do note that the stagnation for $p = 5$ is likely caused by the conditioning issues associated with the least-squares problem that is implied by Choice 3 from Section \ref{sec:local-space-choices}.
\begin{figure}[t]
    \sidecaption[t]
    \includegraphics[width = 0.6\textwidth]{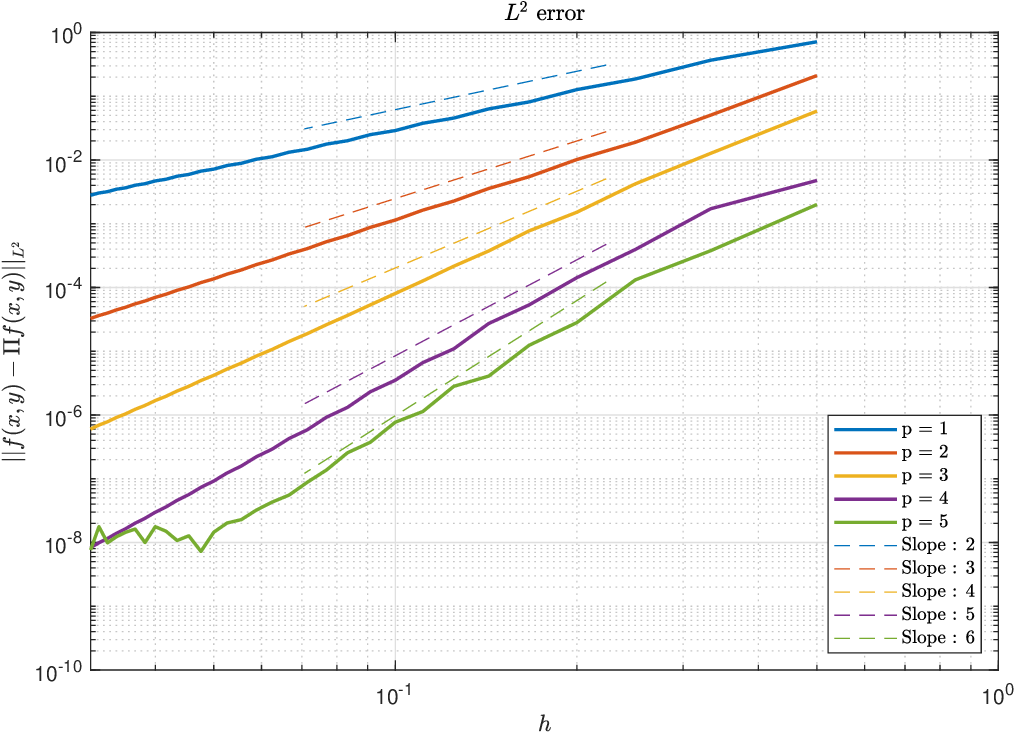}
    \caption{Numerical Convergence rates for projection of $f(x,y) = \sin(\pi x)\sin(\pi y)$ onto the THB-spline space built for the domain hierarchy from \eqref{eq:local-estimates-numerical-problem}. The numerical rates (solid lines) are compared to Theorem \ref{thm:LocalTHBsplineProjEst} (dashed lines).}
    \label{fig:NumResConvergence}
\end{figure}

\subsection{Adaptive refinement tests}
In this section, we compare our projector to the one from \cite{giust_local_2020} on an adaptive refinement test borrowed from the latter.
To do this, we require an adaptive refinement scheme that ensures that the hierarchical refinements satisfy all assumptions required for the formulation of the \Bezier{} projector.
We propose one such adaptive refinement scheme for quadratic and cubic THB-splines for the purpose of conducting this comparison.
Note that the refinement scheme can certainly be improved, however it's not the focus of our paper.

Nonetheless, the numerical results indicate that the scheme shows comparable performance in relation to the results from \cite{giust_local_2020}. For practical purposes, a first implementation of the refinement scheme was formulated for the restricted setting of quadratic and cubic THB-splines, and for a stronger set of assumptions that automatically imply Assumptions \ref{ass:wide-refinements} and \ref{ass:unique-projection-elements}.

\paragraph{\textbf{New assumptions on the mesh}}
The first new assumption is borrowed from \cite{vuong_hierarchical_2011} and it states that the refinement domains $\Omega_\ell$ must be equal to the union of supports of a subset of level-$(\ell-1)$ splines.
\begin{assumpBox}
	\label{ass:RefDomSplSupp}
	For all $\ell > 1$, $\exists S\subset \mathcal{B}_{\ell-1}$ such that
	\begin{eqnarray}
		\Omega_{\ell} := \bigcup_{B_{\vec{j},\ell-1}\in S} \supp\left(B_{\vec{j},\ell-1}\right), \quad \ell = 2,\dots,L.
	\end{eqnarray}
\end{assumpBox}
Assumption \ref{ass:RefDomSplSupp} was originally introduced so that the HB-spline functions  can represent unity with strictly positive coefficients $w_{\vec{j},\ell} > 0$,
\begin{eqnarray}
    \sum_{(\vec{j},\ell) : B_{\vec{j},\ell}\in \mathcal{H}_{\boldsymbol{\Omega}}} w_{\vec{j},\ell} B_{\vec{j},\ell} = 1\;.
\end{eqnarray}

The second assumption is introduced in \cite{mokris_completeness_2014}. For this, we introduce the open complementary region:
\begin{eqnarray}
    \Omega_{\ell+1}^c := \Omega \backslash \Omega_{\ell+1}\;.
\end{eqnarray}
\begin{assumpBox}
\label{ass:SplSuppSimpleConnected}
For each level $\ell = 1,\dots,L-1$, and for any B-spline basis function $B_{\vec{j},\ell}\in \mathcal{B}_{\ell}$, the overlap defined as:
\begin{eqnarray}\label{eq:overlap}
    \overlap\left(B_{\vec{j},\ell}\right):= \supp\left(B_{\vec{j},\ell}\right) \cap \Omega_{\ell+1}^c,
\end{eqnarray}
is connected and simply connected.
\end{assumpBox}
Assumption \ref{ass:SplSuppSimpleConnected} is used in \cite{mokris_completeness_2014} to show that Hierarchical B-spline spaces are linearly independent and contains the space of piecewise polynomials over the same mesh.

The above two assumptions are sufficient for quadratic THB-splines. For the cubic case, we impose the following additional assumption on the refinement domains $\Omega_\ell$.

\begin{assumpBox}
\label{ass:projection-elements-no-overlapping}
When either $p^1 = 3$ or $p^2 = 3$, and for each $\ell = 2,\dots,L$, $\Omega_\ell$ can be described as the following unions:
\begin{itemize}[leftmargin=0.28in]
	\item if $p^1 = 3$ and $p^2 \leq 2$, then $\exists K^1_\ell \subset \NN^2$ such that:
	\begin{eqnarray}
		\Omega_\ell &=& \bigcup_{(k^1,k^2)\in K^1_\ell} \left[\xi^1_{2k^1,\ell-1},\xi^1_{2k^1+2,\ell-1}\right]\times \left[\xi^2_{k^2,\ell-1},\xi^2_{k^2+1,\ell-1}\right];
	\end{eqnarray}
	\item if $p^1 \leq 2$ and $p^2 = 3$, then $\exists K^2_\ell \subset \NN^2$ such that:
	\begin{eqnarray}
		\Omega_\ell &=& \bigcup_{(k^1,k^2)\in K^2_\ell} \left[\xi^1_{k^1,\ell-1},\xi^1_{k^1+1,\ell-1}\right]\times \left[\xi^2_{2k^2,\ell-1},\xi^2_{2k^2+2,\ell-1}\right]\;,
	\end{eqnarray}
	\item if $p^1 = p^2 = 3$, then $\exists K^3_\ell \subset \NN^2$ such that:
	\begin{eqnarray}
		\Omega_\ell &=& \bigcup_{(k^1,k^2)\in K^3_\ell} \left[\xi^1_{2k^1,\ell-1},\xi^1_{2k^1+2,\ell-1}\right]\times \left[\xi^2_{2k^2,\ell-1},\xi^2_{2k^2+2,\ell-1}\right].
	\end{eqnarray}
\end{itemize}
\end{assumpBox}

The following results, proved in the Appendix, show that these assumptions imply Assumptions \ref{ass:wide-refinements} and \ref{ass:unique-projection-elements} required for formulating the \Bezier{} projector.

\begin{proposition}\label{prop:corAss2}
	Assumptions \ref{ass:RefDomSplSupp}, \ref{ass:SplSuppSimpleConnected} and \ref{ass:projection-elements-no-overlapping} imply Assumptions \ref{ass:wide-refinements} and \ref{ass:unique-projection-elements} in two dimensions.
\end{proposition}
\begin{proof}
	The proof is shown in Section \ref{sec:new-mesh-assumptions} appended to the end of the paper.
\end{proof}

\paragraph{\textbf{Adaptive refinement scheme}}
We now propose a simple refinement scheme; Algorithm \ref{alg:MainCodeAdapRef} contains the corresponding pseudo-code.
This algorithm consists of four main routines:
\begin{itemize}
    \item \FuncSty{Project} : builds THB-splines on the given mesh and applies the projector described in Section \ref{sec:THBProj};
    \item \FuncSty{ElemError} : returns the maximum error for every mesh element;
    \item \FuncSty{MarkElem} : marks a proportion of mesh elements for refinement based on D\"orfler marking \cite{dorfler_convergent_1996} and $0<\theta <1$;
    \item \FuncSty{ConformMesh} : updates the set of marked elements so that Assumptions \ref{ass:MeshGrading},\ref{ass:RefDomSplSupp}, \ref{ass:SplSuppSimpleConnected} and \ref{ass:projection-elements-no-overlapping} are satisfied.
\end{itemize}
\IncMargin{1em}
\begin{algorithm}
\SetKwData{Mesh}{mesh}
\SetKwData{f}{f}
\SetKwData{Proj}{proj}
\SetKwData{MarkedElem}{marked\_elem}
\SetKwData{ElemError}{elem\_error}
\SetKwData{MaxError}{max\_error}
\SetKwFunction{Project}{Project}
\SetKwFunction{CalculateMaxElementError}{ElemError}
\SetKwFunction{MarkRefinementElements}{MarkElem}
\SetKwFunction{ConformMeshToAssumptions}{ConformMesh}

\SetKwInOut{Input}{input}\SetKwInOut{Output}{output}
\Input{Target function $f\in L^2(\Omega)$, starting mesh $\Mesh$, target max elem error $\varepsilon > 0$ and constant $0<\theta< 1$. }
\Output{A projection $\Proj$ whos max element error is less than $\varepsilon$. }
\BlankLine
 {$\Proj \leftarrow \Project(f, \Mesh)$}\;
 {$\ElemError \leftarrow \CalculateMaxElementError(f, \Proj, \Mesh)$}\;
\While{$\max(\ElemError) > \varepsilon$}{
 $ \MarkedElem \leftarrow \MarkRefinementElements(\ElemError, \theta)$\;
     $\Mesh \leftarrow \ConformMeshToAssumptions(\MarkedElem, \Mesh)$\;
      $\Proj \leftarrow \Project(f, \Mesh)$\;
     $\ElemError \leftarrow \CalculateMaxElementError(f, \Proj, \Mesh)$\;
}
\caption{Adaptive refinement scheme}
\label{alg:MainCodeAdapRef}
\end{algorithm}

The only non-standard routine here is \FuncSty{ConformMesh}. A pseudo-code for this is shown in Algorithm \ref{alg:ConformMeshToAssumptions} and consists of three routines:
\begin{itemize}
    \item \FuncSty{SupportCover} : for every level $\ell$, initializes $\Omega_\ell$ as a greedy union of supports of $B_{\vec{j},\ell-1}$ that cover the marked level-$\ell$ elements;
    \item \FuncSty{GradeMesh} : conforms mesh to admissibility class $1$ using \cite{bracco_refinement_2018};
    \item \FuncSty{ConnectedSupport} : for any $\ell$ and $B_{\vec{j},\ell-1}\in\mathcal{B}_{\ell-1}$ such that $\overlap(B_{\vec{j},\ell-1})$ is not simply connected, marks the support of $B_{\vec{j},\ell-1}$ for refinement.
\end{itemize}
In addition, for cubic THB-splines, each of these routines is made to conform to Assumption \ref{ass:projection-elements-no-overlapping}.
\IncMargin{1em}
\begin{algorithm}
\SetKwData{Mesh}{mesh}
\SetKwData{MeshNew}{mesh\_new}
\SetKwData{MarkedElem}{marked\_elem}
\SetKwData{Alter}{alter}
\SetKwFunction{SupportCover}{SupportCover}
\SetKwFunction{GradeMesh}{GradeMesh}
\SetKwFunction{SplineSupportConnectedness}{ConnectedSupport}
\SetKwInOut{Input}{input}\SetKwInOut{Output}{output}
\Input{$\Mesh$ to conform to Assumptions \ref{ass:MeshGrading}, \ref{ass:RefDomSplSupp}, \ref{ass:SplSuppSimpleConnected} and \ref{ass:projection-elements-no-overlapping} while adding $\MarkedElem$. }
\Output{$\Mesh$.}
\BlankLine
$\Mesh \gets \SupportCover(\MarkedElem, \Mesh)$\;
 $\Alter \gets$ true\;
\While{$\Alter$}{
     $\MeshNew \gets \GradeMesh(\Mesh)$\;
     $\MeshNew \gets \SplineSupportConnectedness(\MeshNew)$\;
    \If{$\MeshNew = \Mesh$}{
         $\Alter \gets$ false\;
    }
     $\Mesh \gets \MeshNew$\;
}
\caption{Schematic for \FuncSty{ConformMesh}}
\label{alg:ConformMeshToAssumptions}
\end{algorithm}

\paragraph{\textbf{Comparison with \cite{giust_local_2020}}}
The performance of our \Bezier{} projector with the proposed adaptive refinement scheme is compared to the local THB-spline projector introduced in \cite{giust_local_2020}. As a test case, we use the following target function:
\begin{equation}
    \slabel{eq:AdapFunc}f(x,y) = 1 - \tanh\left(\frac{\sqrt{(2x-1)^2+(2y-1)^2 - 0.3}}{0.05\sqrt{2}}\right)\;, \quad (x,y)\in [0,1]^2\;.
\end{equation}
The same function was used in \cite[Example 1]{giust_local_2020} albeit by transforming it from $[0,1]^2$ to a scaled and translated domain $[-1,1]^2$.
This will not cause any problems for our comparison as we will measure the errors in the $L^\infty$ norm.

As in \cite{giust_local_2020}, the quadratic scheme is allowed to refine upto level $5$, while the cubic case refines up to level $4$. In Figure \ref{fig:AdaptiveRefinementP2} and Figure \ref{fig:AdaptiveRefinementP3}, the maximum element error for our refinement scheme can be seen for the quadratic and cubic case for various choices of $\theta$. The final refinement domain for the quadratic case and $\theta = 0.5$ can be seen in Figure \ref{fig:MeshP2Theta50}, and for the cubic case and $\theta=0.75$ in Figure \ref{fig:MeshP3Theta75}.

\begin{figure}[t]
\subfigures
    \includegraphics[width=.45\textwidth]{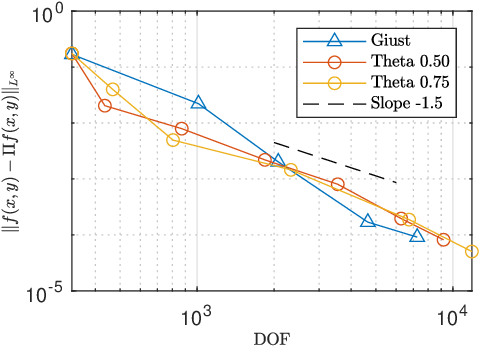}
    \hfill
    \includegraphics[width=.45\textwidth]{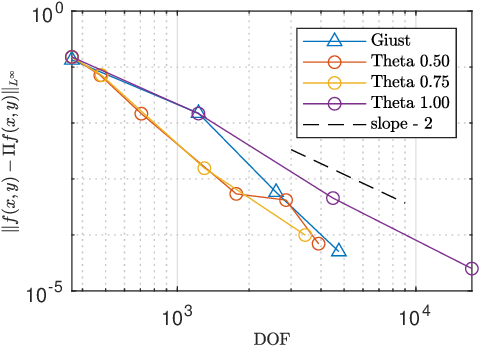}
    \leftcaption{Maximum element error for the quadratic adaptive refinement of equation \eqref{eq:AdapFunc}. \label{fig:AdaptiveRefinementP2}}
    \rightcaption{Maximum element error for the cubic adaptive refinement of equation \eqref{eq:AdapFunc}.\label{fig:AdaptiveRefinementP3}}
\end{figure}

\begin{figure}
\subfigures
    \includegraphics[width=.45\textwidth]{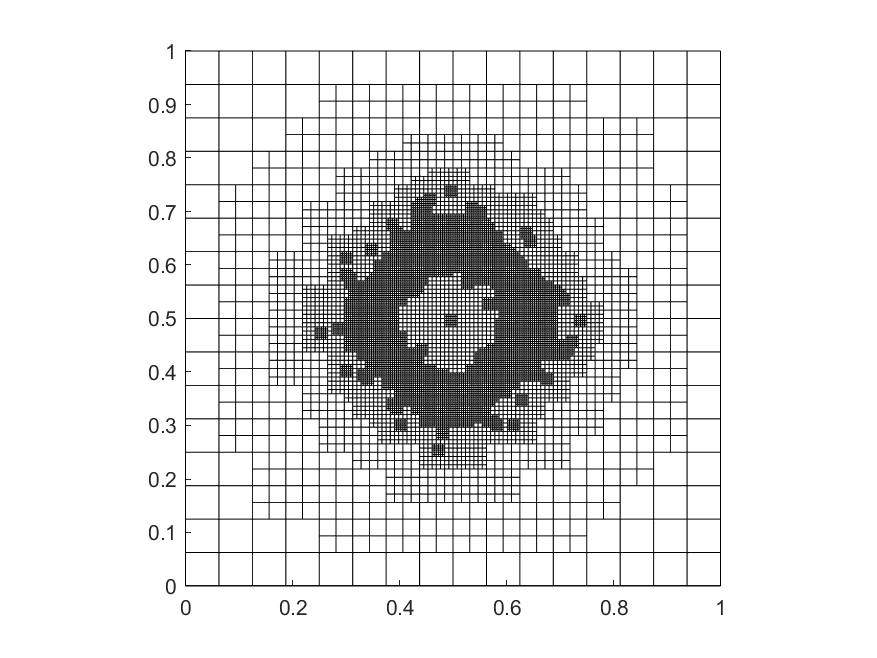}
    \hfill
    \includegraphics[width=.45\textwidth]{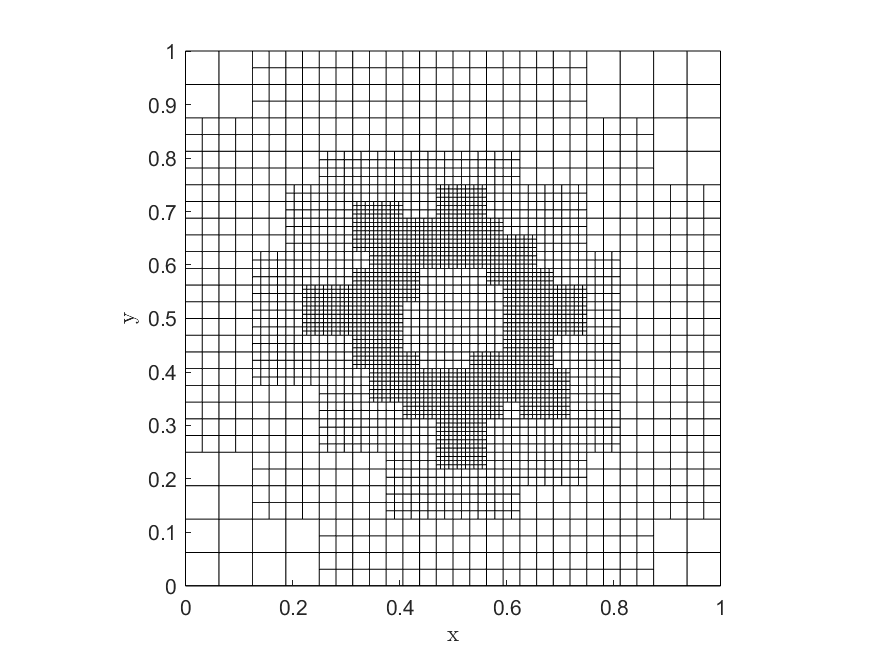}
    \leftcaption{Final quadratic mesh for $\theta =0.5$.\label{fig:MeshP2Theta50}}
    \rightcaption{Final cubic mesh for $\theta =0.75$.\label{fig:MeshP3Theta75}}
\end{figure}

\section{Conclusion}
Truncated Hierarchical B-spline spaces suffer from overloading of mesh elements, i.e., there exist mesh elements restricted to which the basis functions are linearly dependent.
In this paper, we considered these spaces in an arbitrary number of spatial dimensions.
For hierarchically refined meshes that satisfy certain assumptions, we showed that we can identify the elements that are not overloaded.
Using those non-overloaded elements, we partitioned the entire mesh into a set of local macro-elements, none of which are overloaded.
These macro-elements, which we call projection elements, are local in the sense that they consists of adjacent elements, and the number of elements in each macro-element solely depends on the spline degree.

Using these projection elements, we extended the \Bezier{} projector from \cite{Thomas2015} to the Truncated Hierarchical B-spline spaces. In particular, the notion of projection elements is useful for performing local projections onto the spline space restricted to those macro-elements.
These local projections can be chosen in different ways, and we outlined some simple choices that stay close to the original \Bezier{} projector from \cite{Thomas2015}.

Finally, we derived optimal local error estimates for our proposed projector.
These estimates were numerically validated in a two-dimensional setting.
In addition, a first adaptive refinement scheme was proposed for quadratic and cubic splines that produces hierarchically refined meshes conforming to our assumptions.
Note that this refinement scheme was only proposed for performing numerical comparison with the results of \cite{giust_local_2020}; theoretical results (e.g., on its complexity) are outside the scope of this paper.
The results showed that the performance of the \Bezier{} projector and the proposed refinement scheme are comparable to that of \cite{giust_local_2020}.

There are several generalizations of our results that may be of interest in applications, we list a couple of them here.
First, we only considered hierarchical meshes of admissibility class $1$ \cite{bracco_refinement_2018}; extending the characterization of local linear independence to higher admissibility classes will allow one to work with more aggressive hierarchical refinements.
Another interesting open question is formulation of refinement schemes with linear complexity, simple implementations, that can yield meshes conforming to our assumptions, and that work for arbitrary choices of degrees in both 2D and 3D.
These and other questions will be the focus of our future work.

\bibliographystyle{plain} 
\bibliography{references}

\begin{thebibliography}{10}

\bibitem{beirao_da_veiga_estimates_2011}
L.~Beirão~da Veiga, A.~Buffa, J.~Rivas, and G.~Sangalli.
\newblock Some estimates for h-p-k-refinement in {Isogeometric} {Analysis}.
\newblock {\em Numerische Mathematik}, 118(2):271--305, 2011.

\bibitem{bracco_refinement_2018}
Cesare Bracco, Carlotta Giannelli, and Rafael Vázquez.
\newblock Refinement algorithms for adaptive isogeometric methods with
  hierarchical splines.
\newblock {\em Axioms}, 7(3), June 2018.
\newblock Publisher: MDPI AG.

\bibitem{buffa_adaptive_2016}
Annalisa Buffa and Carlotta Giannelli.
\newblock Adaptive isogeometric methods with hierarchical splines: {Error}
  estimator and convergence.
\newblock {\em Mathematical Models and Methods in Applied Sciences},
  26(1):1--25, January 2016.
\newblock arXiv: 1502.00565 Publisher: World Scientific Publishing Co. Pte Ltd.

\bibitem{Cottrell2009}
J~Austin Cottrell, Thomas J~R Hughes, and Yuri Bazilevs.
\newblock {\em Isogeometric analysis}.
\newblock Wiley-Blackwell, Hoboken, NJ, August 2009.

\bibitem{de_boor_calculating_1972}
Carl De~Boor.
\newblock On {Calculating} with {B}-splines.
\newblock {\em Journal of Approximation Theory}, 6(1):50--62, 1972.
\newblock Publication Title: JOURNAL OF APPROXIMATION THEORY Volume: 6.

\bibitem{Dokken2013PolynomialBox-partitions}
Tor Dokken, Tom Lyche, and Kjell~Fredrik Pettersen.
\newblock Polynomial splines over locally refined box-partitions.
\newblock {\em Computer Aided Geometric Design}, 30(3):331--356, 2013.
\newblock Publisher: Elsevier B.V.

\bibitem{dorfler_convergent_1996}
Willy Dörfler.
\newblock A {Convergent} {Adaptive} {Algorithm} for {Poisson}’s {Equation}.
\newblock {\em SIAM Journal on Numerical Analysis}, 33(3):1106--1124, 1996.
\newblock \_eprint: https://doi.org/10.1137/0733054.

\bibitem{forsey_hierarchical_1988}
David~R Forsey and Richard~H Barrels.
\newblock Hierarchical {B}-{Spline} {Refinement}.
\newblock {\em SIGGRAPH Comput. Graph.}, 22(4):205--212, 1988.
\newblock Issue: 4 Publication Title: Computer Graphics Volume: 22.

\bibitem{giannelli_thb-splines_2012}
Carlotta Giannelli, Bert Jüttler, and Hendrik Speleers.
\newblock {THB}-splines: {The} truncated basis for hierarchical splines.
\newblock {\em Computer Aided Geometric Design}, 29(7):485--498, October 2012.
\newblock Issue: 7 ISSN: 01678396.

\bibitem{giust_local_2020}
Alessandro Giust, Bert Jüttler, and Angelos Mantzaflaris.
\newblock Local ({T}){HB}-spline projectors via restricted hierarchical spline
  fitting.
\newblock {\em Computer Aided Geometric Design}, 80:101865, 2020.

\bibitem{greiner_interpolating_1998}
Gunther Greiner and Kai Hormann.
\newblock Interpolating and {Approximating} {Scattered} {3D} {Data} with
  {Hierarchical} {Tensor} {Product} {Splines}.
\newblock {\em Surface Fitting and Multiresolution Methods}, August 1998.

\bibitem{hughes_isogeometric_2005}
T.J.R. Hughes, J.A. Cottrell, and Y.~Bazilevs.
\newblock Isogeometric analysis: {CAD}, finite elements, {NURBS}, exact
  geometry and mesh refinement.
\newblock {\em Computer Methods in Applied Mechanics and Engineering},
  194(39-41):4135--4195, October 2005.

\bibitem{johannessen_isogeometric_2014}
Kjetil~André Johannessen, Trond Kvamsdal, and Tor Dokken.
\newblock Isogeometric analysis using {LR} {B}-splines.
\newblock {\em Computer Methods in Applied Mechanics and Engineering},
  269:471--514, February 2014.

\bibitem{li_s-splines_2019}
Xin Li and Thomas~W. Sederberg.
\newblock S-splines: {A} simple surface solution for {IGA} and {CAD}.
\newblock {\em Computer Methods in Applied Mechanics and Engineering},
  350:664--678, June 2019.
\newblock Publisher: Elsevier B.V.

\bibitem{Lyche2018}
Tom Lyche, Carla Manni, and Hendrik Speleers.
\newblock {\em Foundations of {Spline} {Theory}: {B}-{Splines}, {Spline}
  {Approximation}, and {Hierarchical} {Refinement}}, volume 2219.
\newblock Springer International Publishing, 2018.
\newblock Publication Title: Lecture Notes in Mathematics ISSN: 00758434.

\bibitem{mokris_completeness_2014}
Dominik Mokriš, Bert Jüttler, and Carlotta Giannelli.
\newblock On the completeness of hierarchical tensor-product {B} -splines.
\newblock {\em Journal of Computational and Applied Mathematics}, 271:53--70,
  December 2014.

\bibitem{rainer_kraft_adaptive_1998}
{Rainer Kraft}.
\newblock {\em Adaptive und linear unabh"angige {Multilevel} {B}-{Splines} und
  ihre {Anwendungen}}.
\newblock PhD thesis, Univ. Stuttgard, Stuttgard, 1998.
\newblock Publication Title: PhD Thesis.

\bibitem{Sande2020}
Espen Sande, Carla Manni, and Hendrik Speleers.
\newblock Explicit error estimates for spline approximation of arbitrary
  smoothness in isogeometric analysis.
\newblock {\em Numerische Mathematik}, 144(4):889--929, 2020.
\newblock arXiv: 1909.03559 Publisher: Springer Berlin Heidelberg.

\bibitem{scott_isogeometric_2011}
Michael~A. Scott, Michael~J. Borden, Clemens~V. Verhoosel, Thomas~W. Sederberg,
  and Thomas~J.R. Hughes.
\newblock Isogeometric finite element data structures based on {Bézier}
  extraction of {T}-splines.
\newblock {\em International Journal for Numerical Methods in Engineering},
  88(2):126--156, October 2011.

\bibitem{sederberg_t-splines_2003}
Thomas~W. Sederberg, Jianmin Zheng, Almaz Bakenov, and Ahmad Nasri.
\newblock T-{Splines} and {T}-{NURCCs}.
\newblock {\em ACM Trans. Graph.}, 22(3):477--484, July 2003.
\newblock Place: New York, NY, USA Publisher: Association for Computing
  Machinery.

\bibitem{speleers_hierarchical_2017}
Hendrik Speleers.
\newblock Hierarchical spline spaces: quasi-interpolants and local
  approximation estimates.
\newblock {\em Advances in Computational Mathematics}, 43(2):235--255, April
  2017.

\bibitem{speleers_effortless_2016}
Hendrik Speleers and Carla Manni.
\newblock Effortless quasi-interpolation in hierarchical spaces.
\newblock {\em Numerische Mathematik}, 132(1):155--184, January 2016.

\bibitem{Thomas2015}
D.~C. Thomas, M.~A. Scott, J.~A. Evans, K.~Tew, and E.~J. Evans.
\newblock Bézier projection: {A} unified approach for local projection and
  quadrature-free refinement and coarsening of {NURBS} and {T}-splines with
  particular application to isogeometric design and analysis.
\newblock {\em Computer Methods in Applied Mechanics and Engineering},
  284:55--105, 2015.
\newblock arXiv: 1404.7155 Publisher: Elsevier B.V.

\bibitem{vuong_hierarchical_2011}
A.~V. Vuong, C.~Giannelli, B.~Jüttler, and B.~Simeon.
\newblock A hierarchical approach to adaptive local refinement in isogeometric
  analysis.
\newblock {\em Computer Methods in Applied Mechanics and Engineering},
  200(49-52):3554--3567, December 2011.

\end{thebibliography}

\section{Appendix: Proof of Proposition \ref{prop:corAss2}}\label{sec:new-mesh-assumptions} 

\noindent
\textbf{Assumption \ref{ass:wide-refinements}}:\\
We first show that Assumptions \ref{ass:RefDomSplSupp} and \ref{ass:SplSuppSimpleConnected} imply Assumptions \ref{ass:wide-refinements} in two dimensions.
Here, there are three kinds of well-behaved border elements shown in the schematics below.
\begin{center}
	\begin{tikzpicture}[scale=0.7]
		\draw (0,0) grid (4,4);
		\draw[step=5mm] (1,1) grid (4,4);
		\filldraw[color=orange, opacity = 0.2] (1,1.5) rectangle++ (0.5,2.5);
		\filldraw[color=orange, opacity = 0.2] (1.5,1) rectangle++ (2.5,0.5);
		\filldraw[color=cyan, opacity = 0.5] (1,1) rectangle++ (0.5,0.5);
		\draw[color = red,line width = 1mm] (1,1.5) -- (1,1) -- (1.5,1);
		\draw[color = teal,line width = 1mm] (0,0) rectangle ++ (3,3);
		\draw[color = cyan,thick] (1,1) rectangle ++ (2,2);
		\node[align=center,scale=0.8] at (2,-0.7) {$d_\ell^1(e) = d_\ell^2(e) = 0$};
	\end{tikzpicture}
	\qquad
	\begin{tikzpicture}[scale=0.7]
		\draw (0,0) grid (4,4);
		\draw[step=5mm] (0,1) grid (4,4);
		\draw[step=5mm] (1,0) grid (4,4);
		\filldraw[color=orange, opacity = 0.2] (1,1) rectangle++ (-0.5,3);
		\filldraw[color=orange, opacity = 0.2] (1,1) rectangle++ (3,-0.5);
		\filldraw[color=cyan, opacity = 0.5] (1,1) rectangle++ (0.5,0.5);
		\draw[color = red,line width = 1mm] (1,0.5) -- (1,1) -- (0.5,1);
		\draw[color = teal,line width = 1mm] (0,0) rectangle ++ (3,3);
		\draw[color = cyan,thick] (1,1) rectangle ++ (2,2);
		\node[align=center,scale=0.8] at (2,-0.7) {$d_\ell^1(e)> p^1,~ d_\ell^2(e) > p^2$};
	\end{tikzpicture}
	\qquad
	\begin{tikzpicture}[scale=0.7]
		\draw (0,0) grid (4,4);
		\draw[step=5mm] (1,0) grid (4,4);
		\filldraw[color=orange, opacity = 0.2] (1,2) rectangle++ (3,-0.5);
		\filldraw[color=magenta, opacity = 0.2] (1,2) rectangle++ (0.5,1);
		\filldraw[color=magenta, opacity = 0.2] (1,1.5) rectangle++ (0.5,-1);
		\filldraw[color=cyan, opacity = 0.5] (1,1.5) rectangle++ (0.5,0.5);
		\draw[color = red,line width = 1mm] (1,1.5) -- (1,2);
		\draw[color = teal,line width = 1mm] (0,0) rectangle ++ (3,3);
		\draw[color = cyan,thick] (1,0.5) rectangle ++ (2,2.5);
		\node[align=center,scale=0.8] at (2,-0.7) {$d_\ell^1(e)=0,~ d_\ell^2(e) > p^2$};
	\end{tikzpicture}
\end{center}

For the first two cases (left and centre), the well-behaved border elements shown in blue and the respective domain specified by Assumption \ref{ass:wide-refinements} is shown by blue dashed line. Due to assumption \ref{ass:RefDomSplSupp}, if $d_\ell^i(e) = 0$, the adjacent $2p^i+2$ elements, must be contained in $\Omega_{\ell}$, which are shown in orange.
Then, consider the level $\ell-1$ B-spline with the boundary of its support displayed with bold, teal lines.
Then, the level-$\ell$ elements in its support and fenced by the light blue border must be contained in $\Omega_\ell$ due to Assumption \ref{ass:SplSuppSimpleConnected}, and therefore Assumption \ref{ass:wide-refinements} is satisfied.
For the last case (right), assume without loss of generality that $d_\ell^2(e)>p^2$.
A similar reasoning as above implies that Assumption \ref{ass:SplSuppSimpleConnected} is again satisfied.

\vspace{\baselineskip}
\noindent
\textbf{Assumption \ref{ass:unique-projection-elements}}:\\
We now show that Assumptions \ref{ass:RefDomSplSupp}, \ref{ass:SplSuppSimpleConnected} and \ref{ass:projection-elements-no-overlapping} imply Assumptions \ref{ass:unique-projection-elements} in two dimensions for the quadratic and cubic cases.

Given an element $\Omega^{e'}$ and let $\vec{t}$ be the minimal translation in $l^1$ norm to element of $ \mathcal{M}_\ell\backslash\mathcal{M}_\ell^a$. Without loss of generality, assume $t^i\leq 0$ for all $i$. If both entries $t^i$ are non-zero, then for $\Omega^e = \tau_{\vec{t}+[1,1]^T}(\Omega^{e'})$, we have:
\begin{itemize}
 \item $\partial \Omega^e \cap \left(\partial \Omega_\ell \backslash \partial \Omega \right)\neq \emptyset$
 \item $d_\ell^i(e)>0$ for all $i=1,2$.
\end{itemize}
But then, due to Assumption \ref{ass:projection-elements-no-overlapping}, for both the cubic and quadratic case:
\begin{eqnarray}
 d_\ell^i(e) >p^i,\qquad i = 1,2.
\end{eqnarray}
Hence $\Omega^{e'}$ is contained in the projection element generated by well-behaved border element $\Omega^e$.
In case both entries $t^i$ are zero, $d_\ell^i(e')=0$ and $\Omega^{e'}$ is a well-behaved border element and contained in the projection element $\widehat{\Omega}^{e'}$.
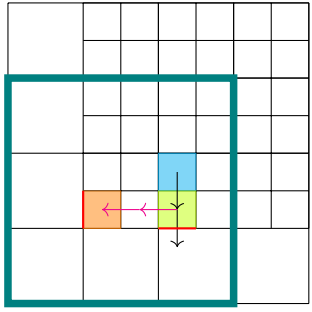
\begin{figure}[t]
 \centering
 \begin{tikzpicture}
      \draw (0,0) grid (4,4);
      \draw[step=5mm] (1,1) grid (4,4);
      \filldraw[color=cyan, opacity = 0.5] (2,1.5) rectangle++ (0.5,0.5);
      \filldraw[color=lime, opacity = 0.5] (2,1) rectangle++ (0.5,0.5);
      \filldraw[color=orange, opacity = 0.5] (1,1) rectangle++ (0.5,0.5);
      \draw[->] (2.25,1.75) -- ++(0,-0.5);
      \draw[->] (2.25,1.25) -- ++(0,-0.5);
      \draw[color = red, thick] (2,1) --++ (0.5,0);
      \draw[->,color=magenta] (2.25,1.25) -- ++(-0.5,0);
      \draw[->,color=magenta] (1.75,1.25) -- ++(-0.5,0);
      \draw[color = red, thick] (1,1) --++ (0,0.5);
      \draw[color = teal, line width = 1mm] (0,0) rectangle ++ (3,3);
  \end{tikzpicture}
 \caption{For the starting element $\Omega^{e'}$ shown in blue, the candidate shown in green does not satisfy $d_\ell^1(e^*)>3=p^1$. Hence, the orange element is the true candidate $\Omega^e$ for which $d_\ell^1(e) = d_\ell^2(e) = 0$. If this were false, the teal B-spline would violate Assumption \ref{ass:SplSuppSimpleConnected}.}
 \label{fig:example4}
\end{figure}
For the last case, assume without loss of generality that $t^1 = 0$. Then the element at translation $\vec{t}_1 = (0,t^2+1)$ is the first candidate. Denote this element as $\Omega^{e^*} = \tau_{\vec{t}_1}(\Omega^{e'})$. If $d_\ell^1(e^*)>p^1$ or $d_\ell^1(e^*)=0$, we are done. Else, the element that $d_\ell^1(e^*)$ points at, denoted by $\Omega^e$, will satisfy $d_\ell^1(e) = 0$ and $d_\ell^2(e) = 0$. Else, one can easily construct a level $(\ell-1)$ B-spline, that violates Assumption \ref{ass:SplSuppSimpleConnected}. The construction of this B-spline is made explicit in Figure \ref{fig:example4}.

Assumption \ref{ass:wide-refinements} is sufficient to show uniqueness for the cases where either $d_\ell^1(e) = 0,~d_\ell^2(e)>p^2$ or $d_\ell^2(e) = 0,~d_\ell^1(e)>p^1$. 
For the remaining cases $d_\ell^1(e)= d_\ell^2(e) = 0$ and $d_\ell^1(e) >p^1,~d_\ell^2(e)>p^2$, note that for two projection elements of the same character (where we mean that both satisfy $d_\ell^1(e)= d_\ell^2(e) = 0$ or $d_\ell^1(e) >p^1,~d_\ell^2(e)>p^2$), they can never overlap by Assumption \ref{ass:RefDomSplSupp}. 
For the mixed case, one of the projection elements will have $d_\ell^{i}(e) = p^i$, and hence be at a minimum $p^i-1$ distance away from the other. This guarantees uniqueness.
\end{document}